\documentclass[a4paper, 12pt]{amsart}
\usepackage{amsmath,amssymb,amsthm,eucal,manfnt}
\usepackage{color, tikz}
\usepackage[all]{xy}

\title[The extension class and KMS states]{The extension class and KMS states for Cuntz--Pimsner algebras of some bi-Hilbertian bimodules}

\author{Adam Rennie}
\author{David Robertson}
\author{Aidan Sims}
\email{renniea@uow.edu.au, dave84robertson@gmail.com, asims@uow.edu.au}
\address{School of Mathematics and Applied Statistics,
University of Wollongong\\
Northfields Ave 2522, Australia}

\thanks{This research was supported by the Australian Research Council.}

\advance\topmargin by -\headheight
\advance\topmargin by -\headsep
\numberwithin{equation}{section} %% needs `amsmath' package

\theoremstyle{plain} %% needs `amsmath' package
\newtheorem{thm}{Theorem}[section]
\newtheorem{lemma}[thm]{Lemma}
\newtheorem{prop}[thm]{Proposition}
\newtheorem{corl}[thm]{Corollary}

\theoremstyle{definition} %% needs `amsmath' package
\newtheorem{defn}[thm]{Definition}
\theoremstyle{remark} %% needs `amsmath' package
\newtheorem{rmk}[thm]{Remark}
\newtheorem{exm}[thm]{Example}

\DeclareMathOperator{\End}{End}   %% endomorphism algebra
\DeclareMathOperator{\Id}{Id}     %% identity map
\DeclareMathOperator{\res}{res}   %%%residue
\DeclareMathOperator{\supp}{supp} %% support
\DeclareMathOperator{\Tr}{Tr}     %% operator trace
\DeclareMathOperator{\opsp}{span}

\newcommand{\op}{{\operatorname{op}}}       %% opposite-algebra marker
\newcommand{\extcls}{[\operatorname{ext}]}

\newcommand{\Aut}{\operatorname{Aut}}

\newcommand{\lsp}{\operatorname{span}}
\newcommand{\clsp}{\operatorname{\mbox{$\overline{\lsp}$}}}
\newcommand{\sr}{r_\sigma}

\newcommand{\B}{\mathcal{B}}  %% another algebra
\newcommand{\C}{\mathbb{C}}   %% complex numbers
\newcommand{\D}{\mathcal{D}}  %% a selfadjoint operator
\renewcommand{\H}{\mathcal{H}}  %% a Hilbert space
\newcommand{\N}{\mathbb{N}}   %% natural numbers
\newcommand{\ox}{\otimes}     %% tensor product
\newcommand{\Oo}{\mathcal{O}}  %%% Cuntz-Pimsner algebra
\newcommand{\R}{\mathbb{R}}   %% real numbers
\newcommand{\Tt}{\mathcal{T}} %%Toeplitz
\newcommand{\T}{\mathbb{T}} %%Toeplitz
\newcommand{\Z}{\mathbb{Z}}   %% integers
\def\pairL_#1(#2|#3){\mbox{${}_{#1}({#2}\mid{#3})$}} %% hermitian pairing _B(s|t)
\def\pairR(#1|#2)_#3{\mbox{$({#1}\mid{#2})_{#3}$}} %% hermitian pairing (s|t)_A
\def\scal<#1|#2>{\langle#1\mid#2\rangle} %% scalar product <y|z>

\newbox\ncintdbox \newbox\ncinttbox %% noncommutative integral symbols
	\setbox0=\hbox{$-$}
	\setbox2=\hbox{$\displaystyle\int$}
	\setbox\ncintdbox=\hbox{\rlap{\hbox
		to \wd2{\hskip-.125em \box2\relax\hfil}}\box0\kern.1em}
	\setbox0=\hbox{$\vcenter{\hrule width 4pt}$}
	\setbox2=\hbox{$\textstyle\int$}
	\setbox\ncinttbox=\hbox{\rlap{\hbox
		to \wd2{\hskip-.175em \box2\relax\hfil}}\box0\kern.1em}

\begin{document}

\begin{abstract}
For bi-Hilbertian $A$-bimodules, in the sense of Kajiwara--Pinzari--Watatani, we
construct a Kasparov module representing the extension class defining the Cuntz--Pimsner
algebra. The construction utilises a singular expectation which is defined using the
$C^*$-module version of the Jones index for bi-Hilbertian bimodules. The Jones
index data also determines a novel quasi-free
dynamics and KMS states on these Cuntz--Pimsner
algebras.
\end{abstract}

\maketitle

%\tableofcontents
%
%
%\parskip=6pt
%\parindent=0pt

%\addtocontents{toc}{\vspace{-1pc}}

\section{Introduction}
The Cuntz--Pimsner algebras introduced in \cite{Pimsner} have attracted enormous
attention over the last fifteen years (see, for example, \cite{AbadieAchigar:RMJM2009,
Carlsen:IJM08, CarlsenOrtega:PLMS2011, DykemaShlyakhtenko:PEMS2001, DykemaSmith:HJM2005,
FowlerMuhlyRaeburn:IUMJ2003, KajPinWat, KaliszewskiQuiggEtAl:JMAA2013, Kumjian:PJM2004,
KwasniewskiLebedev:JFA2013, LacaRaeburnEtAl:JFA2014, LledoVasselli:MN2010,
MuhlySolel:MS98, Schweizer:JFA2001}). They are at once quite tractable and very general,
including models for crossed products and Cuntz--Krieger algebras \cite{Pimsner}, graph
$C^*$-algebras \cite{FowlerLacaEtAl:PAMS00}, topological-graph $C^*$-algebras
\cite{Katsura:TAMS04}, Exel crossed products \cite{BrownloweRaeburn:MPCPS06},
$C^*$-algebras of self-similar actions \cite{Nekrashevych:JOT04} and many others.

Particularly important in the theory of Cuntz--Pimsner algebras is the natural Toeplitz
extension $0 \to \End_A^0(F_E) \to \Tt_E \to \Oo_E \to 0$ of a Cuntz--Pimsner algebra by
the compact endomorphisms of the associated Fock module. For example, Pimsner uses this
extension in \cite{Pimsner} to calculate the $K$-theory of a Cuntz--Pimsner algebra using
that $\End_A^0(F_E)$ is Morita equivalent to $A$ and $\Tt_E$ is $KK$-equivalent to $A$.
It follows that the class of this extension is important in $K$-theory calculations, and
a concrete Kasparov module representing this class could be useful, for example, in
exhibiting Poincar\'e duality for appropriate classes of Cuntz--Pimsner algebras.

When $E$ is an imprimitivity bimodule, this is relatively straightforward (see
Section~\ref{subsec:smeb}) because the Fock representation of $\Tt_E$ is the compression
of a natural representation of $\Oo_E$ on a 2-sided Fock module. But for the general
situation, there is no such 2-sided module. Pimsner sidesteps this issue in
\cite{Pimsner} by replacing the coefficient algebra $A$ with the direct limit $A_\infty$
of the algebras of compact endomorphisms on tensor powers of $E$, and $E$ with the direct
limit $E_\infty$ of the modules of compact endomorphisms from $E^{\otimes n}$ to
$E^{\otimes n+1}$. This is an excellent tool for computing the $K$-theory of $\Oo_E$: the
module $E_\infty$ (rather than $E$ itself) induces the Pimsner--Voiculescu sequence in
$K$-theory, and the Cuntz--Pimsner algebra of $E_\infty$ is isomorphic to that of $E$.
But at the level of $KK$-theory, replacing $E$ with $E_\infty$ changes things
dramatically. The Toeplitz extension associated to $E_\infty$ corresponds to an element
of $KK^1(\Oo_E, A_\infty)$, rather than of $KK^1(\Oo_E, A)$, and the two are quite
different: for example, if $E$ is the 2-dimensional Hilbert space, then $A = \C$, whereas
$A_\infty = M_{2^\infty}(\C)$, the type-$2^\infty$ UHF algebra.

In this paper we consider the situation where $E$ is a finitely generated bi-Hilbertian
bimodule, in the sense of Kajiwara--Pinzari--Watatani \cite{KajPinWat}, over a unital
$C^*$-algebra. Our main result is a construction of a Kasparov-module representative of
the class in $KK^1(\Oo_E, A)$ corresponding to the extension of $\Oo_E$ by $\End_A^0(F_E)$.
We assume our modules $E$ are both full and injective. This situation is quite common,
and we present a range of examples; but much that we do could be extended to more general
finite-index bi-Hilbertian bimodules, \cite{KajPinWat}.

After introducing some basic structural features of the modules we consider in Section
\ref{sec:bimods}, we give a range of examples. We then examine the important special case
of self-Morita equivalence bimodules (SMEBs), which include crossed products by $\Z$.
This case was first calculated by Pimsner \cite{Pimsner} in order to show that $A$ and
$\Tt_E$ are $KK$-equivalent. We present the details here for completeness.

For SMEBs we can produce an unbounded representative of the extension
$$
0\to \End^0_A(F_E)\to \Tt_E\to \Oo_E\to 0
$$
defining $\Oo_E$. Here $E$ is our correspondence, $F_E$ the (positive) Fock space, and
$\Tt_E,\,\Oo_E$ are the Toeplitz--Pimsner and Cuntz--Pimsner algebras, respectively.

Having an unbounded representative can simplify the task of computing Kasparov
products. Since products with the class of this extension define boundary maps
in $K$-theory and $K$-homology exact sequences, this representative is a useful aid
to computing $K$-theory via the Pimsner--Voiculescu exact sequence. An application of
this technique to the quantum Hall effect appears in \cite{BCR}.

For the general case of (finitely generated) bi-Hilbertian bimodules, we do not obtain an
unbounded representative, but the construction of the right $A$-module underlying the
Kasparov module is novel. Using the bimodule structure, we construct a one-parameter
family $\Phi_s:\Tt_E\to A$, $\Re(s)>1$, of positive $A$-bilinear maps. Provided the
residue at $s=1$ exists, we obtain an expectation $\Phi_\infty=\res_{s=1}\Phi_s:\Tt_E\to
A$, which vanishes on the covariance ideal, and so descends to $\Oo_E$. We  use
$\Phi_\infty$ to construct an $A$-valued inner product on $\Oo_E$, and thereby obtain the
underlying $C^*$-module in our $\Oo_E$--$A$-Kasparov module representing the extension
class. We provide a criterion for establishing the existence of the desired residue in
Proposition~\ref{prop:Phi-infty}. We show that this criterion is readily checkable in
some key examples; in particular, we show in Example~\ref{exm:primgraph} that the residue
exists when $E$ is the bimodule associated to a finite primitive directed graph.

The bimodule structure and Jones--Watatani index are essential ingredients in the
construction of $\Phi_\infty$. The (right) Jones--Watatani index also provides a natural
and interesting one-parameter family of quasi-free automorphisms of $\Oo_E$, and we show
that there is a natural family of KMS states on $\Oo_E$ parameterised by
the states on $A$ which are invariant for the dynamics encoded by $E$. This construction
combines ideas from \cite{LacaNeshveyev} and \cite{CNNR}.

There are also corresponding
dynamics arising from the left Jones--Watatani index, and the product of the left and
right indices. The corresponding collections of KMS states would also be interesting, but
we do not address them here. The key point is that many important Cuntz--Pimsner algebras
arise from bi-Hilbertian modules, and this extra structure gives rise to new tools that are worthy of study.

{\bf Acknowledgements.} This work has profited from discussions with Bram Mesland and
Magnus Goffeng. The authors also wish to thank the anonymous referee for several suggestions
which have greatly improved the exposition.

\section{A class of bimodules}
\label{sec:bimods}

Throughout this paper, $A$ will denote a separable, unital, nuclear $C^*$-algebra. Given
a right Hilbert $A$-module $E$ (written $E_A$ when we want to remember the coefficient
algebra), we denote the $C^*$-algebra of adjointable operators by $\End_A(E)$, the
compact endomorphisms by $\End^0_A(E)$ and the finite-rank endomorphisms by
$\End_A^{00}(E)$. The finite-rank endomorphisms are generated
by rank one operators $\Theta_{e,f}$ with $e,\,f\in E$.

\begin{defn}
\label{defn:bimod} Let $A$ be a unital $C^*$-algebra. Following \cite{KajPinWat}, a
bi-Hilbertian $A$-bimodule is a full right $C^*$-$A$-module with inner product
$\pairR(\cdot|\cdot)_A$ which is also a full left Hilbert $A$-module with inner product
$\pairL_A(\cdot|\cdot)$ such that the left action of $A$ is adjointable with respect to
$\pairR(\cdot|\cdot)_A$ and the right action of $A$ is adjointable with respect to
$\pairL_A(\cdot|\cdot)$.
\end{defn}

If $E$ is a bi-Hilbertian $A$-bimodule, then there are two Banach-space norms on $E$,
arising from the two inner-products. The following straightforward lemma shows that these
norms are automatically equivalent.

\begin{lemma}
Let $E$ be a bi-Hilbertian $A$-bimodule. Then there are constants $c, C \in \R$ such that
$\|\pairR(e|e)_A\| \le c\|\pairL_A(e|e)\|$ and $\|\pairL_A(e|e)\| \le C\|\pairR(e|e)_A\|$
for all $e \in E$.
\end{lemma}
\begin{proof}
By symmetry it suffices to find $c$. Suppose that no such $c$ exists. Then there is a
sequence $e_n \in E$ such that $\|\pairR(e_n|e_n)_A\| > n \|\pairL_A(e_n|e_n)\|$. By
normalising, we may assume that each $\|\pairR(e_n|e_n)_A\| = 1$, and hence each
$\|\pairL_A(e_n|e_n)\| < \frac1n$. So $e_n \to 0$ in $E$, and then continuity of
$\pairR(\cdot|\cdot)_A$ forces $\|\pairR(0|0)_A\| = 1$, contradicting the inner-product
axioms.
\end{proof}

Throughout the paper, if we say that $A$ is a finitely generated projective bi-Hilbertian
$A$-bimodule, we mean that it is finitely generated and projective both as a left and as
a right $A$-module.

%\begin{rmk}\label{rmk:AE<->EAo}
%A left Hilbert $A$-module $E$ is the same thing as a right Hilbert $A^\op$-module where
%$A^\op$ is the opposite algebra. More precisely, we can define a right Hilbert
%$A^\op$-module $E^\op = \{e^\op : e \in E\}$ where $e \mapsto e^\op$ is a vector-space
%isomorphism, and define a right action of $A^\op$ on $E^\op$ by $e^\op \cdot a^\op := (a
%\cdot e)^\op$, and an $A^\op$-valued inner product by
%\[
%\pairR(e^\op|f^\op)_{A^\op}:=\pairL_A(e|f)^{*\op}=\pairL_A(f|e)^\op,\qquad e,\,f\in E.
%\]
%Note that this makes the map $e \mapsto e^\op$ an isometric isomorphism of Banach spaces.
%\end{rmk}

The next lemma characterises when a right $A$-module has a left inner product for a
second algebra. It provides a noncommutative analogue of `the trace over the fibres' for
endomorphisms of vector bundles.

For us, a \emph{frame} for a right-Hilbert module $E_A$ is a sequence $(e_i)_{i \in \N}$
of elements such that the series $\sum \Theta_{e_i, e_i}$ converges strictly to
$\operatorname{Id}_E$; note that this would be called a countable right basis in the
terminology of \cite{KajPinWat}, or a standard normalised tight frame in the terminology
of \cite{FrankLarson:JOT2002}. As discussed in \cite[Section~1]{KajPinWat}, every
countably generated Hilbert module $E$ over a $\sigma$-unital $C^*$-algebra $A$ admits a
frame (in our sense), and it admits a finite frame if and only if $\End^0_A(E) =
\End_A(E)$. As discussed in the remark following \cite[Proposition~1.2]{KajPinWat}, if
$(e_i)$ is a frame for $E$, then the net $\sum_{i \in F} \Theta_{e_i, e_i}(f)$ indexed by
finite subsets $F$ of $\{e_i\}$ converges to $f$ for all $f \in E$.

A less general version of the following basic lemma appears in \cite[Lemma 3.23]{LRV}.

\begin{lemma}[cf. {\cite[Lemma~2.6]{KajPinWat}}]\label{lem:Watatani+}
Let $E_A$ be a countably generated right-Hilbert $A$-module, and let $B\subset \End_A(E)$ be a $C^*$-subalgebra.
\begin{enumerate}
\item\label{it:ip->Phi} Suppose that $\pairL_B(\cdot|\cdot)$ is a left $B$-valued
    inner product for which the right action of $A$ is adjointable. Then there is a
    $B$-bilinear faithful positive map $\Phi : \End_A^{00}(E) \to B$ such that
    $\Phi(\Theta_{e,f}) := \pairL_B(e|f)$ for all $e,f \in E$. For any frame $(e_i)$
    for $E$, we have
    \[\textstyle
        \Phi(T) = \sum_i \pairL_B(T e_i|e_i)\quad\text{ for all $T \in \End_A^{00}(E)$.}
    \]
\item\label{it:Phi->ip} Suppose that $\Phi : \End_A^0(E) \to B$ is a $B$-bilinear
    faithful positive map. Then $\pairL_B(e|f) := \Phi(\Theta_{e,f})$ defines a left
    $B$-valued inner product on $E$ for which the right action of $A$ is adjointable.
\end{enumerate}
\end{lemma}
\begin{proof}
(\ref{it:ip->Phi}) Choose a frame $(e_i)$ for $E$. Given a rank-one operator
$\Theta_{e,f}$, using the frame property at the last equality, we calculate:
\begin{align*}
\sum_i \pairL_B(\Theta_{e,f}e_i|e_i)
    &=\sum_i\pairL_B(e{\pairR(f|e_i)_A}|e_i)
    =\sum_i{}\pairL_B(e|e_i{\pairR(e_i|f)_A})\\
    &=\sum_i\pairL_B(e|\Theta_{e_i,e_i}f)
    =\pairL_B(e|f).
\end{align*}
It follows that there is a well-defined linear map $\Phi : \End_A^{00}(E) \to B$ satisfying
$\Phi(\Theta_{e,f}) = \pairL_B(e|f)$ as claimed. The remaining properties of $\Phi$
follow from straightforward calculations. For example,
\[
\Phi(b_1\Theta_{e,f}b_2)=\Phi(\Theta_{b_1e,b_2^*f})=\pairL_B(b_1e|b_2^*f)=b_1\pairL_B(e|f)b_2
=b_1\Phi(\Theta_{e,f})b_2,
\]
so $\Phi$ is $B$-bilinear. Positivity and faithfulness follow from the corresponding
properties of the inner product.

(\ref{it:Phi->ip}) Given $\Phi:\End_A^{0}(E)\to B$, we define
$$
\pairL_B(e|f):=\Phi(\Theta_{e,f}).
$$
Since $(e,f) \mapsto \Theta_{e,f}$ is a left $\End^0_A(E)$-valued inner-product on $E$,
and since $\Phi$ is faithful and $B$-linear, it is routine to check that
$\pairL_B(\cdot|\cdot)$ is positive definite. Since $\Phi$ is positive, it is
$*$-preserving, and so $\pairL_B(e|f) = \pairL_B(f|e)^*$. Write $\varphi$ for the
homomorphism $B \to \End_A(E)$ that implements the left action. Then $B$-linearity of
$\Phi$ gives
\[
b \pairL_B(e|f)
    = b\Phi(\Theta_{e,f})
    = \Phi(\varphi(b)\Theta_{e,f})
    = \Phi(\Theta_{b\cdot e, f})
    = \pairL_B(b\cdot e|f).
\]
So $\Phi$ is a left $B$-valued inner product. For adjointability of the right $A$-action,
observe that
\begin{equation*}
\pairL_B(e\cdot a|f)=\Phi(\Theta_{e\cdot a,f})=\Phi(\Theta_{e,f\cdot a^*})=\pairL_B(e|f\cdot a^*).\qedhere
\end{equation*}
\end{proof}

\begin{rmk}
Unlike the Hilbert space case, the preceding result does not give any automatic cyclicity
properties for the map $\Phi$ (which we might otherwise be tempted to regard as an
operator-valued trace): for $e,f \in E$ and $U\in \End_A(E)$ unitary, we have
$$
\Phi(\Theta_{e,f}U)=\pairL_B(e|U^*f)\quad\mbox{and}\quad\Phi(U\Theta_{e,f})=\pairL_B(Ue|f).
$$
The adjoint $U^*$ in the first expression is the adjoint with respect to the
inner-product $\pairR(\cdot|\cdot)_A$, which is the inverse of $U$. However, it is not
clear that $U^{-1}$ is an adjoint for $U$ with respect to $\pairL_B(\cdot|\cdot)$, even
assuming that $U$ is adjointable for $\pairL_B(\cdot|\cdot)$.
\end{rmk}

\begin{rmk}
Consider the (very) special case where $A$ is commutative, $E$ is a symmetric $A$-bimodule in
the sense that $a\cdot e = e\cdot a$ for all
$e\in E$, and $\pairL_A(e|f)=\pairR(f|e)_A$. Then the operator-valued weight associated to
$\pairL_A(\cdot|\cdot)$ is a trace: given $\Theta_{e,f}$ and
$\Theta_{g,h}$,
$$
\Phi(\Theta_{e,f}\Theta_{g,h})=\Phi(\Theta_{e(f|g)_A,h})=\pairL_A(e{\pairR(f|g)_A}|h)
$$
and
$$
\Phi(\Theta_{g,h}\Theta_{e,f})=\pairL_A(g{\pairR(h|e)_A}|f).
$$
The following computation shows that these are equal.
\begin{align*}
\pairL_A(e{\pairR(f|g)_A}|h)
    &=\pairR(h|e{\pairR(f|g)_A})_A
    =\pairR(h|e)_A\pairR(f|g)_A
    =\pairR(f{\pairR(e|h)_A}|g)_A\\
    &=\pairL_A(g|f{\pairR(e|h)_A})
    =\pairL_A(g{\pairR(h|e)_A}|f).
\end{align*}
Thus for vector bundles we recover the trace over the fibres of endomorphisms.
\end{rmk}

\begin{rmk}
If $T\in \End_A^{0}(E)$ commutes with all $b\in B$ then $\Phi(T)\in \mathcal{Z}(B)$, because
$$
b\Phi(T)=\Phi(bT)=\Phi(Tb)=\Phi(T)b.
$$
\end{rmk}

\subsection{Examples} We devote the remainder of this section to showing that
many familiar and important classes of correspondences give rise to bi-Hilbertian
bimodules of the type we consider.

\subsubsection{Self-Morita equivalence bimodules (SMEBs)}
The following examples all share the important property that both the left and right
endomorphism algebras are isomorphic to the coefficient algebra (or its opposite). This
will turn out to be an important hypothesis, and also covers many important examples.

\begin{defn}
Let $A$ be a $C^*$-algebra. A self-Morita equivalence bimodule (SMEB) over $A$ is a
bi-Hilbertian $A$-bimodule $E$ whose inner products are both full and satisfy the
imprimitivity condition
$$
\pairL_A(e|f)g=e\pairR(f|g)_A,\quad\mbox{ for all}\ e,\,f,\,g\in E.
$$
\end{defn}

Recall from \cite[Proposition 3.8]{RaeburnW} that if $E_A$ is a self-Morita equivalence
$A$-bimodule, then $A \cong \End^0_A(E)$.

\begin{exm}[Crossed products by $\Z$]
\label{ex:cross}
Suppose that $A$ is unital and nuclear, and let $\alpha:A\to A$ be an automorphism. Then
the $C^*$-correspondence ${}_\alpha A_A$ with the usual right module structure, left
action of $A$ determined by $\alpha$ and left inner product
$\pairL_A(a|b)=\alpha^{-1}(ab^*)$ is a SMEB. The imprimitivity condition follows from the
calculation
\[
a \cdot \pairR(b|c)_A
    = a b^* c
    = \alpha(\alpha^{-1}(ab^*))c
    = \pairL_A(a|b)\cdot c.
\]
\end{exm}

\begin{exm}[Line bundles]
\label{ex:line}
Suppose that $A$ is unital and commutative, so that $A\cong C(X)$ for some
second-countable compact Hausdorff space $X$. Given a complex line bundle $L\to X$, we
obtain a SMEB over $A$ by setting $E=\Gamma(L)$, the continuous sections of $L$. The left
and right actions are by pointwise multiplication, and any Hermitian form $\langle \cdot,
\cdot\rangle$ on $L$ determines inner products by $\pairL_A(e|f)(x) := \langle e(x),
f(x)\rangle =: \pairR(f|e)_A$.
\end{exm}

The next result shows that for SMEBs, the map $\Phi$ of Lemma \ref{lem:Watatani+} is an expectation.

\begin{lemma}
Suppose that $E$ is a SMEB over a unital $C^*$-algebra $A$. The map $\Phi : \End^0_A(E)
\to A$ of Lemma~\ref{lem:Watatani+}(\ref{it:ip->Phi}) satisfies $\Phi(\Id_E) = 1_A$.
\end{lemma}
\begin{proof}
Choose a frame $(e_i)$ for $E$. Since $E$ is a SMEB, \cite[Proposition 3.8]{RaeburnW}
says that the map $\Theta_{x,y} \mapsto \pairL_A(x|y)$ determines an isomorphism $\psi :
\End^0_A(E) \to A$. In particular, $\psi$ is unital, and so
\[
1_A = \psi(\Id_E)
    = \psi\big(\sum_i \Theta_{e_i, e_i}\big)
    = \sum_i \pairL_A(e_i|e_i)
    = \Phi(\Id_E).\qedhere
\]
\end{proof}

Conversely, Corollary~4.14 of \cite{KajPinWat} shows that a bi-Hilbertian bimodule $E$
satisfies $\Phi(\Id_E)=1_A$ if and only $E$ can be given a left inner product which
makes $E$ into a SMEB.

\subsubsection{Crossed products by injective endomorphisms} \label{exm:injectiveendos}

Let $A$ be a unital $C^*$-algebra and suppose that $\alpha : A \to A$ is an injective
unital $*$-endomorphism. Assume there exists a faithful conditional expectation $\Phi : A
\to \alpha(A)$. Then $L := \alpha^{-1}\circ \Phi$ is a \emph{transfer operator}
\cite[Definition 2.1]{Exel2000}; that is, $L: A \to A$ is a positive linear map
satisfying
\[
 L(\alpha(a)b) = aL(b)
\]
for all $a,b \in A$.

There is a bi-Hilbertian $A$-bimodule associated to the triple $(A,\alpha,L)$ as follows:
$A$ is a pre-Hilbert $A$-bimodule with
\[
 a \cdot e \cdot b := ae\alpha(b)
\]
and
\[
 (e|f)_A := L(e^*f)
\]
for $a,b,e,f \in A$. Denote by $E$ the completion of $A$ for the norm
$\|e\|^2 = \|(e|e)_A\|$. By faithfulness of $\Phi$, there is a left inner-product
\[
 {}_A(e|f) = ef^*
\]
which is left $A$-linear and for which the right action of $A$ on $E$ is adjointable. The
associated Cuntz--Pimsner algebra satisfies
\[
 \Oo_E = A \rtimes_{\alpha,L} \N
\]
where $A \rtimes_{\alpha,L} \N$ is as defined by Exel \cite{Exel2000}.

\subsubsection{Vector bundles}

If $E\to X$ is a complex vector bundle over a compact Hausdorff space, then the $C(X)$-module
$\Gamma(E)$ of all continuous sections under pointwise multiplication
is finitely generated and projective for any nondegenerate
$C(X)$-valued inner products (left and right). If we alter the left action
by composing with an automorphism, we
also need to alter the left inner product as in Example \ref{ex:cross}. If $E$ is rank one then
we are back in the SMEB case of Example \ref{ex:line}.

\subsubsection{Topological graphs}
\label{ex:top-graph}

A topological graph is a quadruple $G = (G^0,G^1,r,s)$ where $G^0, G^1$ are locally
compact Hausdorff spaces, $r: G^1 \to G^0$ is a continuous map and $s: G^1 \to G^0$ is a
local homeomorphism. For simplicity, we will assume that $r$ and $s$ are surjective.
Given a topological graph $G$, Katsura \cite{Katsura:TAMS04} associates a right Hilbert
module as follows.  Let $A = C_0(G^0)$. Then, similarly to Section~\ref{sub:CK-algs},
$C_c(G_1)$ is a right pre-Hilbert $A$-module with left and right actions
$$
(a\cdot e\cdot b)(g):=a(r(g))\,e(g)\,b(s(g)),\quad e\in C_c(G^1),\ \ a, b\in A,\ \ g\in G^1
$$
and inner product
$$
(e|f)_A(v)=\sum_{s(g)=v}\overline{e(g)}f(g),\qquad e, f\in C_c(G^1),\ \ v\in G^0.
$$
(Since $s$ is a local homeomorphism, $\{g \in vG^1 : e(g) \not= 0\}$ is finite for $e \in
C_c(G^1)$, so this formula for the inner-product makes sense.) We write $E$ for the
completion of $C_c(G^1)$ in the norm determined by the inner-product, and $E$ is a right
Hilbert $A$-module.

To impose a left Hilbert module structure on $E$, we restrict attention to topological
graphs where $r$ is also a local homeomorphism, and define
\[
 \pairL_A(e|f)(v) = \sum_{r(g)=v} e(g) \overline{f(g)},\qquad e, f\in C_c(G^1),\ \ v\in G^0.
\]

For the remainder of this section, suppose that $G^0$ and $G^1$ are compact. The
following lemma and its proof are due to  Mitch Hawkins, \cite{Mitch}.

\begin{lemma}\label{lem:Gmn}
Suppose that $r,\,s:G^1\to G^0$ are local homeomorphisms.
For each $n \in \N$, the sets $\{v \in G^0 : |G^1v| = n\}$ and $\{w \in G^0 : |vG^1| =
n\}$ are compact open.
\end{lemma}
\begin{proof}
We show that $\{v \in G^0 : |G^1v| = n\}$ is compact open; symmetry does the rest. It
suffices to show that $\{v \in G^0 : |G^1v| \ge n\}$ is both closed and open.

First suppose that $v$ satisfies $|G^1v| \ge n$. Fix distinct $e_1, \dots, e_n \in G^1 v$.
Since $G^1$ is Hausdorff, we can pick disjoint open neighbourhoods $U_i$ of $e_i$. Since
$s$ is a local homeomorphism, we can shrink the $U_i$ so that $s(U_i) = s(U_j)$ for all
$i,j \le n$ and so that $s|_{U_i}$ is a homeomorphism for each $i$. Since $s$ is a local
homeomorphism, it is an open map, and so $V = s(U_1)$ is an open neighbourhood of $v$. For
each $v' \in V$ each $U_iv'$ is a singleton, and the $U_i$ are mutually disjoint, so
$|G^1v'| \ge n$. Hence $V \subseteq \{v \in G^0 : |G^1v| \ge n\}$, and we deduce that the
latter is open.

We now show that it is also closed. Suppose that $v_m$ is a sequence in $G^0$ converging
to $v$, and suppose that each $|G^1v_m| \ge n$. For each $m$, choose distinct elements
$e_{m,1}, \dots, e_{m,n}$ of $G^1v_m$. Since $G^1$ is compact, by passing to a
subsequence we may assume that each sequence $e_{m,i}$ converges to some $e_i \in G^1$.
By continuity of $s$, we have $s(e_i) = v$ for each $i$, so it suffices to show that $i
\not= j$ implies $e_i \not= e_j$. For this, fix a neighbourhood $U$ of $e_i$ on which $s$
is a homeomorphism. Since $e_{m,i} \to e_i$, the $e_{m,i}$ eventually belong to $U$.
Since each $s(e_{m,j}) = v_j = s(e_{m,i})$ and each $e_{m,j} \not= e_{m,i}$, we see that
$e_{m,j} \not\in U_i$ for large $m$. Since $e_{m,j} \to e_j$, we deduce that $e_j \not\in
U$, and in particular $e_j \not= e_i$.
\end{proof}

\begin{corl}
For $m,n \in \N$, let
\[
G^1_{m,n} := \{e \in G^1 : |r(e)G^1| = m\text{ and }|G^1s(e)| = n\}.
\]
Then the $G^1_{m,n}$ are compact open sets, as are $s(G^1_{m,n})$ and $r(G^1_{m,n})$.
\end{corl}
\begin{proof}
We have $G^1_{m,n} = s^{-1}(\{v : |G^1v| = n\}) \cap r^{-1}(\{w : |wG^1| = m\})$.
Lemma~\ref{lem:Gmn} and continuity of $s$ and $r$ imply that $G^1_{m,n}$ is clopen; since
$G^1$ is compact, each $G^1_{m,n}$ then also compact. Since $r, s$ are local
homeomorphisms, they are open maps, so $r(G^1_{m,n})$ and $s(G^1_{m,n})$ are open. They
are compact as they are continuous images of the compact set $G^1_{m,n}$.
\end{proof}

Since $r,s$ are local homeomorphisms, each edge $e$ has a neighbourhood $U_e$ on which
both $s$ and $r$ are homeomorphisms. By the preceding corollary, we may assume that each
$U_e \subseteq G^1_{|r(e)G^1|, |G^1s(e)|}$. The $U_e$ cover $G^1$, so by compactness,
there is a finite open cover $\mathcal{U}$ such that each $U \in \mathcal{U}$ is
contained in some $G^1_{m(U), n(U)}$. Choose a partition of unity on $G^1$ subordinate to
$\mathcal{U}$; say $\{f_U : U \in \mathcal{U}\}$. So $0 \le f_U \le 1$ and $f_U
\in C_0(U)$ for each $U \in \mathcal{U}$, and $\sum_{U \in \mathcal{U}} f_U(e) = 1$ for
all $e \in G^1$.

\begin{lemma}
For each $U \in \mathcal{U}$, define $h_U \in C(G^1)$ by $h_U(e) := \sqrt{f_U(e)}$. The
collection $\{h_U : U \in \mathcal{U}\}$ is a frame for both the left and the right
inner-product on $C(G^1)$. We have $\Phi(\Id_{E})(v) = |vG^1|$ for all $v \in G^0$.
\end{lemma}
\begin{proof}
The situation is completely symmetrical in $r$ and $s$, so we just have to show that the
$f_U$ form a frame for the right inner-product. For this, we fix $g \in C(G^1)$ and $e
\in G^1$ and calculate
\begin{equation}\label{eq:fUcalc}
\sum_{U \in \mathcal{U}} (\theta_{h_U, h_U} g)(e)
    = \sum_U h_U(e) \pairR(h_U|g)_{C(G^0)}(s(e))
    = \sum_U \sum_{s(e') = s(e)} \sqrt{f_U(e)} \overline{\sqrt{f_U(e')}} g(e')
\end{equation}
Since $s$ restricts to a homeomorphism on each $U \in \mathcal{U}$, we can only have
$f_U(e)$ and $f_U(e')$ simultaneously nonzero in the sum on the right-hand side
of~\eqref{eq:fUcalc} if $e = e'$. Since $f_U$ is real-valued, we have
$\overline{\sqrt{f_U(e)}} = \sqrt{f_U(e)}$, and so
\[
\sum_{U \in \mathcal{U}} (\theta_{h_U, h_U} g)(e)
    = \sum_{U \in \mathcal{U}} h_U(e)^2 g(e)
    = \Big(\sum_U f_U(e)\Big)g(e)
    = g(e).
\]
This proves that the $h_U$ constitute a frame. For the final assertion, we calculate
\begin{align*}
\Phi(\Id_{E})(v)
    &= \sum_U \pairL_{C(G^0)}(h_U|h_U)(v)
    = \sum_U \sum_{r(e) = v} h_U(e)\overline{h_U(e)}\\
    &= \sum_{r(e) = v} \sum_U f_U(e)
    = \sum_{r(e) = v} 1
    = |vG^1|.\qedhere
\end{align*}
\end{proof}

\subsubsection{Cuntz--Krieger algebras}
\label{sub:CK-algs} As a specific case of the example above, suppose  that
$G=(G^0,G^1,r,s)$ is a finite directed graph where $G^0$ and $G^1$ both have the discrete
topology. We suppose for simplicity that $G$ has no sources and no sinks. If $B$ is the
edge-adjacency matrix of $G$, then the Cuntz--Pimsner algebra $\Oo_E$ of the right
Hilbert $A$-module $E_A$ is the Cuntz--Krieger algebra $O_B$ \cite[Example~2,
page~193]{Pimsner}. If we think of the left Hilbert $A$-module ${}_AE$ as a right Hilbert
$A^\op$ module $E^\op_{A^\op}$ with $e^\op \cdot a^\op = (a \cdot e)^\op$ and $(e^\op
\mid f^\op)_{A^\op} = \big({}_A(f \mid e)\big)^\op$, then the Cuntz--Pimsner algebra
$\Oo_{E^\op}$ is the Cuntz--Krieger algebra $O_{B^t}$ given by the transpose of the
matrix $B$, which is given by the graph $G^\op$ defined by reversing the edges of $G$.

\subsubsection{Twisted topological graphs}
\label{sub:twistedtop-graph}

The following construction is due to Li \cite{Li2014}. Suppose that $G = (G^0,G^1,r,s)$
is a topological graph. Let $N = \{ N_\alpha : \alpha \in \Lambda \}$ be an open cover of
$G^1$. Given $\alpha_1, \dots, \alpha_n \in \Lambda$, write $N_{\alpha_1\dots\alpha_n} =
\bigcap_{i=1}^n N_{\alpha_i}$. A collection of functions
\[
 S = \{s_{\alpha\beta} \in C(\overline{N_{\alpha\beta}}, \mathbb{T}) : \alpha,\beta\in \Lambda \}
\]
is called a \emph{$1$-cocycle relative to $N$} if $s_{\alpha\beta}s_{\beta\gamma} =
s_{\alpha\gamma}$ on $\overline{N_{\alpha\beta\gamma}}$.

Let
\[
 C_c(G,N,S) := \Big\{ x \in \prod_{\alpha\in\Lambda} C(\overline{N_\alpha}) :
        x_\alpha = s_{\alpha\beta} x_\beta \textrm{ on } \overline{N_{\alpha\beta}}
        \textrm{ and } \overline{x_\alpha}x_\alpha \in C_c(E^1) \Big\}
\]
For $x \in C_c(G, N, S)$ and $g \in G^1$, we write $x(g)$ for the tuple
$\big(x(g)_\alpha\big)_{\alpha \in \Lambda}$, with the convention that $x(g)_\alpha = 0$
when $g \not\in N_\alpha$. Choose for each $g \in E^1$ an element $\alpha(g)$ such that
$g \in N_{\alpha(g)}$; since the $s_{\alpha\beta}$ are circle valued, for $x,y \in C_c(G,
N, S)$, the map $g \mapsto \overline{x(g)_{\alpha(g)}}y(g)_{\alpha(g)}$ does not depend
on our choice of the assignment $g \mapsto \alpha(g)$. Now $C_c(G,N,S)$ becomes a
pre-right-Hilbert $C_0(G^0)$-module under the operations
\begin{align*}
 (x\cdot a)(g)_\alpha &= x(g)_\alpha a(s(g)), \\
 \pairR(x|y)_A(v) &= \sum_{s(g)=v} \overline{x(g)}_{\alpha(g)} y(g)_{\alpha(g)},\text{ and} \\
 (a\cdot x)(g)_\alpha &= a(r(g))x(g)_\alpha
\end{align*}
for $x,y \in C_c(G,N,S)$, $a\in A$, $\alpha\in\Lambda$ and $v\in G^0$. Theorem~3.3 of
\cite{Li2014} ensures that the completion $E(G,N,S)$ of $C_c(G,N,S)$ is a right-Hilbert
$A$-bimodule.

If $r : G^1 \to G^0$ is a local homeomorphism, then there is a left inner-product on
$E(G,N,S)$ satisfying
\[
 \pairL_A(x|y)(v) = \sum_{r(g)=v} x(g)_{\alpha(g)}\overline{y(g)}_{\alpha(g)},
\]
which again does not depend on our choice of assignment $g \mapsto \alpha(g)$. The right
action is adjointable for this left inner-product, and $E(G,N,S)$ is then a bi-Hilbertian
$A$-bimodule.

\section{A Kasparov module representing the extension class}

We now show how to represent the Kasparov class arising from the defining extension of a
Cuntz--Pimsner algebra of a bimodule. The easy case turns out to be the SMEB case, which
we treat first, since in this case we can also obtain more information
in the form of an unbounded representative of the
Kasparov module.

The SMEB case does not immediately show how to proceed in the general case: the dilation of
the representation-mod-compacts of $\Oo_E$ on the Fock module
to an actual representation of $\Oo_E$ is
easily achieved in the SMEB case by using a two-sided Fock module.

Utilising the extra information coming from the bi-Hilbertian bimodule structure  allows us to
handle the general case, by constructing an $A$ module with a representation of $\Oo_E$.

\subsection{The SMEB case}
\label{subsec:smeb}
The following theorem summarises the situation when $\Phi({\Id_E})=1_A$. The bounded
Kasparov module representing the extension in this case is implicit in Pimsner \cite{Pimsner}, and
numerous similar constructions of Kasparov modules associated to circle actions have
appeared in \cite{AKL,CNNR,MesCocyle,PaskRennie:JFA2006} amongst others.
Similar results for the unbounded Kasparov module were obtained in \cite{GG}.
The Fock module associated to $C^*$-bimodules $E$ over a
$C^*$-algebra $A$ is defined as
\[
F_E:=\bigoplus_{n\geq 0}E^{\ox n}
\]
with $E^{\otimes 0} := A$, where the internal product $E^{\ox n}$
is taken regarding $E$ as a right $A$-module with
a left action of $A$.
We let $\extcls$
denote the class of the extension
$$
0\to \End^0_A(F_E)\to T_E\to \Oo_E\to 0
$$
in $KK^1(\Oo_E, \End^0_A(E))$, and $[F_E]\in KK(\End^0_A(F_E),A)$ the class of the Morita
equivalence.

\begin{thm}
Let $E$ be a SMEB over $A$. For $\Z\ni n < 0$, define $E^{\ox n} := \overline{E}^{\ox |n|}$.
Let $F_{E, \Z}$ denote the Hilbert-bimodule direct sum
\[
F_{E,\Z}:=\bigoplus_{n\in\Z} E^{\ox n},
\]
and define an operator $N$ on the algebraic direct sum $\bigcup^\infty_{n=1}
\bigoplus^n_{m=-n} E^{\ox m} \subseteq F_{E, \Z}$ by $N\xi := n\xi$ for $\xi \in E^{\ox
n}$. There is a homomorphism $\rho : \Oo_E \to \End_A(F_{E, \Z})$ such that $\rho(s_e)\xi
= e \otimes \xi$ for all $e \in E$ and $\xi \in \bigcup E^{\otimes n}$.  The triple
\[
(\Oo_E,(F_{E,\Z})_A,N)
\]
is an unbounded Kasparov $\Oo_E$--$A$ module that represents the class
$\extcls\ox_{\End^0_A(F_E)}[F_E] \in KK^1(\Oo_E, A)$.
\end{thm}
\begin{proof}
If $E$ is a SMEB, then the conjugate module $\overline{E}$ is also a SMEB, and we have
$E \otimes_A \overline{E} \cong \End^0_A(E) \cong A$,
and similarly $\overline{E} \otimes_AE \cong A$. This shows that the coefficient algebra $A$ is the
fixed point algebra for the gauge action, and that the spectral subspaces for the
gauge action are full. Then by \cite[Proposition 2.9]{CNNR},
$(\Oo_E,(F_{E,\Z})_A,N)$ is an unbounded Kasparov module.

The corresponding bounded Kasparov module is determined by the non-negative spectral
projection of $N$, denoted $P_+$, \cite[Section 7]{Kas}. Since $P_+ F_{E, \Z}$ is
canonically isomorphic to $F_E$ and compression by $P_+$ implements a positive splitting
for the quotient map $q : \Tt_E \to \Oo_E$, we deduce that $(\Oo_E, F_{E, \Z}, 2P_+-1)$
represents $\extcls$, and hence $(\Oo_E, F_{E, \Z}, N)$ does too.
\end{proof}

\subsection{An operator-valued weight}

Our next goal is to construct a Kasparov module representing the extension class in the
case when $E$ is not a SMEB. To do so, we seek to dilate the Fock representation of
$\Tt_E$ to a representation of $\Oo_E$, but we cannot do this using the module $F_{E,\Z}$
above when $E$ is not a SMEB; the 2-sided direct sum  does not carry a
representation of $\Oo_E$ by translation operators. In \cite{Pimsner}, Pimsner
circumvents this problem by enlarging $E$ to a module $E_\infty$ over the core
$\Oo_E^\gamma$, and enlarging the Fock module accordingly. This has the disadvantage,
however, that the resulting exact sequence
$$
0 \to \End_{\Oo_E^\gamma}^0(F_{E_\infty})
\to \Tt_{E_\infty} \to \Oo_{E_\infty} \cong \Oo_E \to 0
$$
is very different from
the sequence $0 \to \End_A^0(F_E) \to \Tt_E \to \Oo_E \to 0$ in which we were originally
interested. For example, if $A = \C$ and $E = \C^2$, then $\End_A^0(F_E) \cong
\mathbb{K}$, whereas $\End_{\Oo_E^\gamma}(F_{E^\infty})$ is Morita equivalent to the UHF
algebra $M_{2^\infty}(\C)$.

In this subsection, we show how to dilate the Fock representation
without changing coefficients when $E$ is a finitely generated bi-Hilbertian bimodule
satisfying an analytic hypothesis, and present some examples of this situation.
This will require some set-up building on the tools developed in
Section~\ref{sec:bimods}. We construct the desired Kasparov module in
subsection \ref{subsec:kas-mod}.

Fix a bi-Hilbertian $A$-bimodule $E$, and let $\{e_i\}$ be a frame for the right module $E_A$. Given a
multi-index $\rho=(\rho_1,\dots,\rho_k)$ we write $e_\rho=e_{\rho_1}\ox\cdots\ox
e_{\rho_k}$ for the corresponding element of the natural frame of $E^{\ox k}$. We define
\begin{equation}
e^{\beta_k}=\sum_{|\rho|=k}\,\pairL_A(e_\rho|e_\rho)=\Phi_k(\Id_{E^{\ox k}}),
\label{eq:beta-k}
\end{equation}
and just write $e^\beta$ for $e^{\beta_1}$. Provided that $E_A$ is full and finitely generated,
$e^{\beta_k}$ is a positive, invertible and central element of $A$, \cite{KajPinWat}, so that $\beta_k\in
A$ is the logarithm of $\Phi_k(\Id_{E^{\ox k}})$. Since $E$ will always be clear from context,
this justifies our notation
\begin{equation}
\beta:=\log(\Phi({\rm Id}_E)),\qquad \beta_k:=\log(\Phi({\rm Id}_{E^{\otimes k}})).
\label{eq:log-beta}
\end{equation}
We write
$\rho=\underline{\rho}\overline{\rho}$ for the decomposition of a multi-index $\rho$ into
its initial and final segments, whose lengths will be clear from context. From the
formula
$$
\pairL_A(e_\rho|e_\rho)
    =\pairL_A(e_{\underline{\rho}}\,{\pairL_A(e_{\overline{\rho}}|e_{\overline{\rho}})}|e_{\underline{\rho}})
$$
we see that for $0\leq n\leq k$
$$
e^{\beta_k}
=\sum_{|\underline{\rho}|=k-n}\pairL_A(e_{\underline{\rho}} e^{\beta_{n}}|e_{\underline{\rho}})
\leq \Vert e^{\beta_n}\Vert\,e^{\beta_{k-n}}.
$$

\begin{lemma}
\label{lem:correct-beta1} Let $E$ be a finitely generated bi-Hilbertian $A$-bimodule and
for $k\geq0$, define $\Phi_k:\,\End_A(E^{\ox k})\to A$ by
\begin{equation}
\Phi_k(T):=\sum_{|\rho|=k}\pairL_A(Te_\rho|e_\rho),
\label{eq:op-valued-KMS1}
\end{equation}
where $\pairL_A(\cdot|\cdot)$ is the left $A$-valued inner product on $F_E$. For $n\leq
k$ and $\xi, \eta \in E^{\ox n}$ we have
$$
\Phi_k(\Theta_{\xi,\eta}\ox \Id_{k-n})
    = \pairL_A(\xi|\eta e^{\beta_{(k-|\eta|)}}).
$$
\end{lemma}

\begin{proof}
We calculate, using centrality of $e^{\beta_k}$ in $A$ at the fifth step,
\begin{align*}
\Phi_k(\Theta_{\xi,\eta}\ox \Id_{k-n})&=\sum_{|\rho|=k}\,\pairL_A(\Theta_{\xi,\eta}e_\rho|e_\rho)
=\sum_{|\rho|=k}\,\pairL_A(\xi\cdot{\pairR(\eta|e_{\underline{\rho}})_A}\,e_{\overline{\rho}}|e_\rho)\\
&=\sum_{|\rho|=k}\,\pairL_A(\xi\cdot{\pairR(\eta|e_{\underline{\rho}})_A}\,
{\pairL_A(e_{\overline{\rho}}|e_{\overline{\rho}})}|e_{\underline{\rho}})
=\sum_{|\underline{\rho}|=|\eta|}\,\pairL_A(\xi\cdot{\pairR(\eta|e_{\underline{\rho}})_A}\,
e^{\beta_{(k-|\eta|)}}|e_{\underline{\rho}})\\
&=\sum_{|\underline{\rho}|=|\eta|}\,\pairL_A(\xi\cdot e^{\beta_{(k-|\eta|)}}{\pairR(\eta|e_{\underline{\rho}})_A}
|e_{\underline{\rho}})
=\sum_{|\underline{\rho}|=|\eta|}\,\pairL_A(\xi\cdot e^{\beta_{(k-|\eta|)}}|
e_{\underline{\rho}}\cdot{\pairR(e_{\underline{\rho}}|\eta)_A})\\
&=\pairL_A(\xi\cdot e^{\beta_{(k-|\eta|)}}|\eta)
=\pairL_A(\xi|\eta\cdot e^{\beta_{(k-|\eta|)}}).\qedhere
\end{align*}
\end{proof}

\begin{lemma}
\label{lem:Phiinfty}
Let $E$ be a finitely generated bi-Hilbertian $A$-bimodule, and for each $k\geq 0$, let
$\Phi_k :\End^0(E^{\otimes k}) \to A$  be the positive map of Lemma~2.4(1).
For $0 \le T \in \Tt_E$,
and for $\Re(s) > 1$, the series
\begin{equation}\label{eq:Phi s}
\sum^\infty_{k=0} \Phi_k(T)e^{-\beta_k}(1 + k^2)^{-s/2}
\end{equation}
converges to an element $\Phi^s_\infty(T)$ of $A$ which is positive for $s$ real.
\end{lemma}
\begin{proof}
By definition, we have $\Phi_k(\Id_{E^{\ox k}})e^{-\beta_k}=1_A$. Thus for $\Re(s)>1$,
the series
$$
\sum_{k=0}^\infty \Phi_k(\Id_{E^{\ox k}})e^{-\beta_k}(1+k^2)^{-s/2}
    = \Big(\sum^\infty_{k=0} (1 + k^2)^{-s/2}\Big) 1_A
$$
converges in norm. The function $s \mapsto \sum \Phi_k(\Id_{E^{\otimes k}})e^{-\beta k}
(1 + k^2)^{-s/2}$ has well-defined residue $1_A$ at $s = 1$. Since $\Tt_E$ can be
regarded as a subalgebra of $\End_A(F_E)$, the formula~\eqref{eq:op-valued-KMS1} makes
sense for $T \in \Tt_E$. Indeed, if $P_k : F_E \to E^{\otimes k}$ is the projection, then
$P_k \End_A(F_E) P_k \cong \End_A(E^{\otimes k})$, and then
$\Phi_k(T) = \Phi_k(P_k TP_k)$ for all $T \in \Tt_E$.

For $0\leq T\in \Tt_E$, the inequality $T\leq \Vert T\Vert \Id_{\Tt_E}$ shows that
\begin{equation}
\Phi_k(T)e^{-\beta_k}\leq \Vert T\Vert\,1_A.
\label{eq:beta-inequality}
\end{equation}
So for $s \in \C$ with $\Re(s) > 1$, the series~\eqref{eq:Phi s} converges in norm.
\end{proof}

We now construct an $A$-valued map on $\Tt_E$ by taking the residue at $s = 1$ of the map
$\Phi^s_{\infty}$ of Lemma~\ref{lem:Phiinfty}, and then show that this
residue functional factors through
$\Oo_E$.

Recall that, given sequences $(x_n),\,(y_n)$ of real numbers, we write
$x_n \in O(y_n)$ if there is a constant $C$ such that $x_n \leq
Cy_n$ for large $n$.

\begin{lemma}
Let $E$ be a finitely generated bi-Hilbertian $A$-bimodule, take $\eta \in
E^{\ox n}$ and suppose that the sequence $\big(e^{-\beta_k}\eta
e^{\beta_{(k-|\eta|)}}\big)_k$ converges; write $\tilde\eta$ for its limit. Suppose
that there exists $\delta>0$ such that
$$
\Vert e^{-\beta_k}\eta e^{\beta_{(k-|\eta|)}}-\tilde\eta\Vert \in O(k^{-\delta}).
$$
For $\xi\in F_E$ the function $s \mapsto \Phi^s_\infty(T_\xi T^*_\eta)$ has a
well-defined residue $\Phi_\infty(T_\xi T^*_\eta)$ at $s = 1$, and we have
$\Phi_\infty(T_\xi T^*_\eta) = \pairL_A(\xi|\tilde\eta)$, where the inner-product is
taken in $F_E$.
\end{lemma}
\begin{proof}
For $k > |\eta|$, and for a multi-index $\rho = \underline{\rho}\overline{\rho}$ of
length $k$ with $|\underline{\rho}| = |\eta|$, we have $T_\xi T^*_\eta e_\rho =
\Theta_{\xi,\eta}(e_{\underline{\rho}})\otimes e_{\overline{\rho}}$. So for $k > |\eta|$,
\begin{align}
\Phi_k(T_\xi T_\eta^*)
    &=\delta_{|\xi|,|\eta|} \sum_{|\rho| = k}
        \pairL_A(\Theta_{\xi,\eta}(e_{\underline{\rho}})\otimes e_{\overline{\rho}}
            | e_\rho) \nonumber\\
    &=\delta_{|\xi|,|\eta|} \sum_{|\rho| = k}
        \pairL_A(\Theta_{\xi,\eta}(e_{\underline{\rho}})
                    \cdot {\pairL_A(e_{\overline{\rho}} | e_{\overline{\rho}})}
                 | e_{\underline{\rho}}) \nonumber\\
    &= \delta_{|\xi|,|\eta|} \sum_{\underline{\rho}}
        \pairL_A(\Theta_{\xi,\eta}(e_{\underline{\rho}})
                    \cdot e^{-\beta_{k-|\eta|}} | e_{\underline{\rho}}).\label{eq:stop}
\end{align}
Since $e^{-\beta_{k - |\eta|}}$ is central and self-adjoint, we have
\[
\Theta_{\xi,\eta}(e_{\underline{\rho}}) \cdot e^{-\beta_{k-|\eta|}}
    = \xi \cdot \pairR(\eta|e_{\underline{\rho}})_A e^{-\beta_{k-|\eta|}}
    = \xi \cdot e^{-\beta_{k-|\eta|}}\pairR(\eta|e_{\underline{\rho}})_A
    = \xi \cdot \pairR(\eta\cdot e^{-\beta_{k-|\eta|}}|e_{\underline{\rho}})_A.
\]
So~\eqref{eq:stop} gives
\[
\Phi_k(T_\xi T_\eta^*)
    = \delta_{|\xi|,|\eta|} \Phi_{|\eta|}(\Theta_{\xi,\eta\cdot e^{\beta_{(k-|\eta|)}}})
    = \delta_{|\xi|,|\eta|} \pairL_A(\xi|\eta\cdot e^{\beta_{(k-|\eta|)}}).
\]
Since the summands of $F_E$ are, by definition, mutually orthogonal in its inner product,
we deduce that for any $\xi,\eta$ and $k$, we have
\begin{align*}
\Phi_k(T_\xi T_\eta^*)e^{-\beta_k}
    &=\chi_{\{1, \dots, k\}}(|\eta|)
        \pairL_A(\xi|\eta\cdot e^{\beta_{(k-|\eta|)}})e^{-\beta_k}\\
    &=\chi_{\{1, \dots, k\}}(|\eta|)
        \pairL_A(\xi|\tilde\eta)
        + \pairL_A(\xi|e^{-\beta_k}\cdot\eta\cdot e^{\beta_{(k-|\eta|)}}-\tilde\eta),
\end{align*}
and so $\|\Phi_k(T_\xi T^*_\eta)e^{-\beta k} - \pairL_A(\xi|\tilde\eta)\| \in
O(k^{-\delta})$. In particular, $\res_{s=1} \Phi^s_\infty(T_\xi T^*_\eta)$ exists and is
equal to $\pairL_A(\xi|\tilde\eta)$ as claimed.
\end{proof}

\begin{prop}
\label{prop:Phi-infty} Let $E$ be a finitely generated bi-Hilbertian $A$-bimodule.
Suppose that for every $\eta\in F_E$ the limit $\tilde\eta:=\lim_{k\to\infty}
e^{-\beta_k}\eta e^{\beta_{(k-|\eta|)}}$ exists and that for each $\eta$ there is a
$\delta$ such that
$$
\Vert e^{-\beta_k}\eta e^{\beta_{(k-|\eta|)}}-\tilde\eta\Vert \in O(k^{-\delta}).
$$
Then there is a conditional expectation $\Phi_\infty : \Tt_E \to A$ such that
\[
\Phi_\infty(T_\xi T^*_\eta) = \res_{s=1} \Phi^s_\infty(T_\xi T^*_\eta),
\]
and this $\Phi_\infty$ descends to a well-defined functional $\Phi_\infty:\Oo_E\to A$.
\end{prop}
\begin{proof}
For $T\in \End^0_A(E^{\otimes k})$ self-adjoint, we have
$$
\Phi_k(T)e^{-\beta_k} = \Phi_k(P_k T
P_k)e^{-\beta_k} \le \Phi_k(\|T\|P_k) e^{-\beta_k} = \|T\|.
$$
So for a self-adjoint
finite sum $\sum_i T_{\xi_i} T^*_{\eta_i}$, $\|\sum_i\Phi^s_\infty(T_{\xi_i}
T^*_{\eta_i})\| \le \|\sum_i T_{\xi_i} T^*_{\eta_i}\| \sum^\infty_{k = 0}(1 +
k^2)^{-s/2}$. Since every $T$ can be expressed as a sum of two self-adjoints, we deduce
that $\Phi_\infty$ is bounded on $\lsp\{T_\xi T^*_\eta : \xi,\eta \in F_E\}$. It
follows that $\Phi_\infty$ extends by linearity to a bounded linear map on $\lsp\{T_\xi
T^*_\eta : \xi,\eta \in F_E\}$, and so extends to $\Tt_E$.

It is routine to check from the defining formula that $\Phi_\infty$ is positive,
idempotent and $A$-linear.

For the last assertion, we compute:
\begin{align*}
\Phi_\infty(a-\sum_j\Theta_{e_j,e_j}a)
&=\res_{s=1}\sum_{k=0}^\infty
(a-\sum_j\pairL_A(e_j|e_j\cdot e^{\beta_{(k-1)}})e^{-\beta_k}a)(1+k^2)^{-s/2}\\
&=\res_{s=1}\sum_{k=0}^\infty (a-e^{\beta_{k}}e^{-\beta_k}\cdot a)(1+k^2)^{-s/2}\\
&=\res_{s=1}\sum_{k=0}^\infty (a-a)(1+k^2)^{-s/2}=0.
\end{align*}
Hence $\Phi_\infty$ vanishes on the covariance ideal, and so descends to the
quotient $\Oo_E$.
\end{proof}

\begin{exm}[Cuntz algebras] Fix $N \geq 1$. Let $E$ be the Hilbert space
$\mathbb{C}^N = \opsp\{e_i : 1\leq i \leq N \}$, which is a bi-Hilbertian $\C$-bimodule
in the obvious way. Then $\Oo_E \cong \Oo_N$. We have
$e^{\beta_k} = N^k$ for $k\geq 1$. If $\eta\in E^{\otimes n}$ and $k \geq n$ then
\[
 e^{-\beta_k} \cdot\eta\cdot e^{\beta_{k-n}} = N^{-k} \eta N^{k-n} = N^{-n} \eta
\]
and so the hypotheses of Proposition~\ref{prop:Phi-infty} are satisfied and $\Phi_\infty$
exists. In fact, we have
\[
 \Phi_\infty(S_\eta S_\zeta^*) = \frac{1}{N^{|\eta|}}\delta_{\eta,\zeta},
\]
and $\Phi_\infty$ is the usual KMS state for the gauge action on $\Oo_N$.
\end{exm}

For the next example, we need to recall a bit of Perron--Frobenius theory and state an
elementary lemma. If $A$ is a primitive nonnegative matrix, then the Perron--Frobenius
theorem (see, for example \cite[Chapter~8]{Meyer}) says that $A$ has a unique eigenvector
$x$ with nonnegative entries and unit $2$-norm. The entries of $x$ are in fact all
strictly positive, and $A x = \sr(A) x$ where $\sr(A)$ is the spectral radius of $A$. (We
avoid the usual notation, $\rho(A)$, for the spectral radius because the symbol $\rho$ is
used extensively as a multi-index elsewhere in the paper.) The sequence $\sr(A)^{-k} A^k$
converges in norm to the rank-one projection $P$ onto $\C x$, which commutes with $A$.
The following elementary lemma describes the rate of convergence of this sequence.
\begin{lemma}
\label{lem:rate} Let $A$ be a primitive nonnegative matrix, $x$ its $2$-norm-unimodular
Perron--Frobenius eigenvector, and $P$ the projection onto $\C x$. Then there exist  $C >
0$ and $\alpha < 1$ such that $\|\sr(A)^{-k} A^k - P\| \le C \alpha^k$ for all $k$.
\end{lemma}
\begin{proof}
Since $P$ commutes with $A$, we have $A^k = PA^kP +
(1-P)A^k(1-P)=(PAP)^k+((1-P)A(1-P))^k$ for all $k$. Since $PA^kP = \sr(A)^k P$ for all
$k$, we then have $\sr(A)^{-k} A^k - P = \sr(A)^{-k}(1-P)A^k(1-P)$. So $\|\sr(A)^{-k}A^k
- P\| = \sr(A)^{-k}\|(1-P)A^k(1-P)\|$ for all $k$. Let $\lambda := \sr((1-P)A(1-P))$.
Then $\lambda$ is an eigenvalue of $A$ and hence the Perron--Frobenius theorem gives
$|\lambda| < \sr(A)$. The spectral-radius formula then gives $\|(1-P)A^k(1-P)\|^{1/k} \to
|\lambda| < \sr(A)$, and so there exists $l$ such that $\|(1-P)A^l(1-P)\| < \sr(A)^l$. So
$\alpha := \|(1-P)A^l(1-P)\|^{1/l}\sr(A)^{-l} < 1$. For every $k$, we have
$\sr(A)^{-kl}\|(1-P)A^{kl}(1-P)\| \le \sr(A)^{-kl}\|(1-P)A^l(1-P)\|^k < \alpha^{kl}$. Now
for any $p < l$, we have
\begin{align*}
\sr(A)^{-kl - p}\|(1-P)A^{kl + p}(1-P)\|
    &\le \sr(A)^{-p}\|(1-P)A^p(1-P)\| \alpha^{kl} \\
    &= \sr(A)^{-p}\alpha^{-p}\|(1-P)A^p(1-P)\| \alpha^{kl + p},
\end{align*}
and so any $C \ge \max_{p < l} \sr(A)^{-p}\alpha^{-p}\|(1-P)A^p(1-P)\|$ does the job.
\end{proof}

\begin{exm}[$C^*$-algebras of primitive graphs]\label{exm:primgraph}
Fix a finite primitive directed graph $G$ and let $E(G)$ be the associated
$\mathbb{C}^{G^0}$-module. Write $A = A_G \in M_{G^0}(\mathbb{Z})$ for the vertex
adjacency matrix of $G$. For $k\geq 1$ we have
\[
 e^{\beta_k} = \sum_{v\in G^0} |vG^k|\delta_v = \sum_{v,w\in G^0} A^k(v,w) \delta_v
\]
Fix $\lambda \in G^n$. We write $\delta_\lambda = \delta_{\lambda_1} \otimes \dots
\otimes \delta_{\lambda_n} \in E(G)^{\otimes n}$. For $k>n$ we have
\[
 e^{-\beta_k}\cdot\delta_\lambda\cdot e^{\beta_{k-n}}
    = \frac{|s(\lambda)G^{k-n}|}{|r(\lambda)G^k|} \delta_\lambda
    = \frac{\|(A^t)^{k-n}\delta_{s(\lambda)} \|_1}{\|(A^t)^k\delta_{r(\lambda)} \|_1}\delta_\lambda
\]
We show first that
$
 \lim_{k\to\infty} e^{-\beta_k}\cdot\delta_\lambda\cdot e^{\beta_{k-n}}$
exists. Since $G$ is a finite primitive directed graph, we can apply Perron--Frobenius
theory to the transpose $A^t$ of its vertex-adjacency matrix. Let $x \in
\mathbb{C}^{G^0}$ be the $2$-norm-unimodular Perron--Frobenius eigenvector for $A^t$; so
$A^tx = \sr(A^t)x$, and $\|x\|_2 = 1$. Let $P$ be the projection onto $\C x$. By
Lemma~\ref{lem:rate}, there exist $\alpha < 1$ and $C > 0$ such that $\|\sr(A^t)^{-k}
(A^t)^k - P\| \le C\alpha^k$ for all $k$. Since $x$ has real entries, we have
\[
\lim_{k \to \infty} \sr(A^t)^{-k} (A^t)^k \delta_v = P\delta_v = x \langle \delta_v, x\rangle = x_v x
\]
for each $v \in E^0$. Hence
\begin{align}
 \lim_{k\to\infty} e^{-\beta_k} \cdot \delta_\lambda \cdot e^{-\beta_{k-n}}
    &= \lim_{k\to\infty} \frac{\sum_{v \in E^0} A^{k-n}(s(\lambda), v)}{\sum_{w \in E^0} A^k(r(\lambda), w)} \delta_\lambda \nonumber\\
    &= \lim_{k\to \infty} \frac{\sr(A^t)^{k-n}}{\sr(A^t)^k}\left( \frac{\|\lim_{k\to\infty}\frac{1}{\sr(A^t)^{k-n}}(A^t)^{k-n}\delta_{s(\lambda)}\|_1}
                                                    {\|\lim_{k\to\infty}\frac{1}{\sr(A^t)^k}(A^t)^k \delta_{r(\lambda)}\|_1} \right)\delta_\lambda \label{eq:kth term}\\
    &= \frac{1}{\sr(A^t)^n}\frac{x_{s(\lambda)}}{x_{r(\lambda)}}\delta_\lambda.\nonumber
\end{align}

To calculate the rate of convergence, we use Equation~\eqref{eq:kth term}  to write
\[
\Big|\sr(A^t)^n\Big(e^{-\beta_k} \cdot \delta_\lambda \cdot e^{-\beta_{k-n}} - \frac{1}{\sr(A^t)^n}\frac{x_{s(\lambda)}}{x_{r(\lambda)}}\delta_\lambda\Big)\Big|
    = \Bigg|\frac{\|\sr(A^t)^{n-k}(A^t)^{k-n} \delta_{s(\lambda)}\|_1}{\|\sr(A^t)^{-k}(A^t)^k\delta_{r(\lambda)}\|_1} - \frac{\|x_{s(\lambda)} x\|_1}{\|x_{r(\lambda)}x\|_1}\Bigg|.
\]
We have $\|\sr(A^t)^{-k}(A^t)^k\delta_{r(\lambda)} - x_{r(\lambda)} x\|_1 \le C\alpha^k
\|x\|_1$ for all $k$. For $k\in\N$ we have
\[
(1 - C\alpha^k) \|x_{r(\lambda)} x\|_1
    \le \|\sr(A^t)^{-k}(A^t)^k\delta_{r(\lambda)}\|_1
    \le (1 + C\alpha^k) \|x_{r(\lambda)} x\|_1,
\]
and hence
\begin{align*}
(1 + C\alpha^k)^{-1} \frac{\|\sr(A)^{n-k}(A^t)^{k-n} \delta_{s(\lambda)}\|_1}{\|x^G_{r(\lambda)} x^G\|_1}
    &\le \frac{\|\sr(A)^{n-k}(A^t)^{k-n} \delta_{s(\lambda)}\|_1}{\|\sr(A)^{-k}(A^t)^k\delta_{r(\lambda)}\|_1} \\
    &\le (1 - C\alpha^k)^{-1} \frac{\|\sr(A)^{n-k}(A^t)^{k-n} \delta_{s(\lambda)}\|_1}{\|x^G_{r(\lambda)} x^G\|_1}.
\end{align*}
Consequently
\begin{align*}
\Bigg|&\frac{\|\sr(A^t)^{n-k}(A^t)^{k-n} \delta_{s(\lambda)}\|_1}{\|\sr(A^t)^{-k}(A^t)^k\delta_{r(\lambda)}\|_1} - \frac{\|x_{s(\lambda)} x\|_1}{\|x_{r(\lambda)}x\|_1}\Bigg|\\
     &\hskip2em\le \max\Bigg\{\Bigg|\frac{(1 + C\alpha^k)^{-1}\|\sr(A^t)^{n-k}(A^t)^{k-n} \delta_{s(\lambda)}\|_1 - \|x_{s(\lambda)} x\|_1}{\|x_{r(\lambda)} x\|_1}\Bigg|,\\
     &\hskip7em \Bigg|\frac{(1 - C\alpha^k)^{-1}\|\sr(A^t)^{n-k}(A^t)^{k-n} \delta_{s(\lambda)}\|_1 - \|x_{s(\lambda)} x\|_1}{\|x_{r(\lambda)} x\|_1}\Bigg|\Bigg\}.
\end{align*}
Using the identity $(1 + C\alpha^k)^{-1} = 1 - C\alpha^k(1 + C\alpha^k)^{-1}$, we see
that
\begin{align*}
\big|(1 + C\alpha^k)^{-1}&\|\sr(A^t)^{n-k}(A^t)^{k-n} \delta_{s(\lambda)}\|_1 - \|x_{s(\lambda)} x\|_1\big|\\
    &\le \big|\|\sr(A^t)^{n-k}(A^t)^{k-n} \delta_{s(\lambda)}\|_1 - \|x_{s(\lambda)} x\|_1|\big| \\
    &\qquad + \big|C\alpha^k(1 + C\alpha^k)^{-1} \|\sr(A^t)^{n-k}(A^t)^{k-n} \delta_{s(\lambda)}\|_1\big|.
\end{align*}
The first term is in $O(\alpha^k)$ by the reverse triangle inequality and
Lemma~\ref{lem:rate}. The second term is in $O(\alpha^k)$ because the sequences $(1 +
C\alpha^{-k})^{-1}$ and $\|\sr(A^t)^{n-k}(A^t)^{k-n} \delta_{s(\lambda)}\|_1$ are
convergent, and hence bounded. Similar estimates show that $\big|(1 -
C\alpha^k)^{-1}\|\sr(A^t)^{k-n}(A^t)^{n-k} \delta_{s(\lambda)}\|_1 - \|x_{s(\lambda)}
x\|_1\big|$ is in $O(\alpha^k)$. Hence
\[
\Big\|e^{-\beta_k}\cdot\delta_\lambda\cdot e^{\beta_{k-n}} - \frac{1}{\sr(A^t)^n}\frac{x_{s(\lambda)}}{x_{r(\lambda)}}\delta_\lambda\Big\|
     \in O(\alpha^k).
\]
Since the $\delta_\lambda$ span $E^{\otimes n}$, it follows that $\|e^{-\beta_k}\cdot
\eta \cdot e^{\beta_{k-n}} - \tilde\eta\| \in O(\alpha^k)$ for each $\eta \in E^{\otimes
n}$. Since every $\delta > 1$ satisfies $k^{-\delta} > \alpha^k$ for large $k$, it
follows that the module $E$ satisfies the hypotheses of Proposition~\ref{prop:Phi-infty}.
\end{exm}

\begin{rmk}
Since the Cuntz--Krieger algebra of a $\{0,1\}$-matrix $A$ is isomorphic to the graph
$C^*$-algebra of the graph with adjacency matrix $A$ \cite[Proposition~4.1]{kprr}, the
preceding example shows that Proposition~\ref{prop:Phi-infty} covers the situation of
Cuntz-Krieger algebras associated to primitive matrices.
\end{rmk}

The following example is the graph $C^*$-algebraic realisation of the $C^*$-algebra of
$SU_q(2)$ \cite{HongSzymanski:cmp02}. Since the graph in question is not primitive, the
analysis of the preceding example does not apply, but we can check the hypotheses of
Proposition~\ref{prop:Phi-infty} by hand.

\begin{exm}
Consider the following graph $G$.
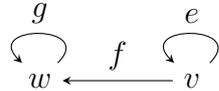
\[\begin{tikzpicture}
    \node (v) at (2,0) {$v$};
    \node (w) at (0,0) {$w$};
    \draw[-stealth] (v)--(w) node[pos=0.5, above]{$f$};
    \draw[-stealth] (v) .. controls +(0.75,0.75) and +(-0.75,0.75) .. (v) node[pos=0.5, above] {$e$};
    \draw[-stealth] (w) .. controls +(0.75,0.75) and +(-0.75,0.75) .. (w) node[pos=0.5, above] {$g$};
\end{tikzpicture}\]
The module $E$ is a copy of $\C^3$ which we write as $\lsp\{\delta_e, \delta_f,
\delta_g\}$. The left action of the projection $p_v$ is by $1$ on $\delta_e$ and zero
elsewhere, and $p_w=1-p_v$. The right action has $p_v$ acting by $1$ on both $\delta_e$
and $\delta_f$ and by zero on $\delta_g$. Schematically,
$$
E=\begin{pmatrix}\delta_e\\ \delta_f\\ \delta_g\end{pmatrix},\quad L(p_v)=\begin{pmatrix} 1 & 0 & 0\\
0 & 0 & 0\\ 0 & 0 & 0\end{pmatrix},\quad R(p_v)=\begin{pmatrix} 1 & 0 & 0\\
0 & 1 & 0\\ 0 & 0 & 0\end{pmatrix}.
$$
Hence $e^\beta=p_v+2p_w$. We have
$$
E^{\ox n}=\begin{pmatrix} \delta_{e^{\ox n}}\\ \delta_{f\ox e^{\ox n-1}}\\ \delta_{g\ox f\ox e^{\ox n-2}}\\
\vdots\\ \delta_{g^{\ox n-1}\ox f}\\ \delta_{g^{\ox n}}\end{pmatrix}
$$
and the left action of $p_v$ is nonzero only on $\delta_{e^{\ox n}}$, while the right
action of $p_w$ is nonzero only on $g^{\ox n}$. Hence $e^{\beta_n}=p_v+(n+1)p_w$. So for
a path $\lambda\in G^n$
$$
e^{-\beta_k}\delta_\lambda e^{\beta_{(k-n)}}=\left\{\begin{array}{ll} \delta_\lambda & \lambda=e^{n}\\
\frac{k+1-n}{k+1}\delta_\lambda & \lambda=g^{n}\\ \frac{1}{k+1}\delta_\lambda & \mbox{otherwise}.\end{array}\right.
$$
Hence the hypotheses of Proposition~\ref{prop:Phi-infty} are satisfied and $\Phi_\infty$
is well-defined for the algebra of this graph.
\end{exm}

\begin{rmk}
The boundedness of the sequence $\Phi_k(T)e^{-\beta_k}$, for $T\in T_E$, suggests that we could extend
the definition of $\Phi_\infty$ using Dixmier trace methods. The difficulty is with the
meaning of $\omega$-limits
in the $C^*$-algebra $A$.
Victor Gayral has pointed out to us that in any representation $\pi:A\to \B(\H)$, for any vectors
$\xi,\,\eta\in\H$, and for a suitable generalised limit $\omega\in (L^\infty([0,\infty)))^*$, the functional
$$
T\mapsto \omega\text{-}\!\lim_N\left(\sum_{k=0}^N\langle \xi,\Phi_k(T)e^{-\beta_k}\eta\rangle(1+k^2)^{-1/2}\right)
$$
is well-defined and so we can make sense of `weak $\omega$-limits'. Unfortunately, however,
the resulting limits in general lie in $A''$ rather than $A$, so they are not well suited
to our purposes.
\end{rmk}

In the special case where $\beta f = f \beta$ for all $f\in E$, the limits above always
exist.

\begin{lemma}
\label{lem:symm-compute}
If $\beta f=f\beta$ for all $f\in E$, then
\[
e^{\beta_n}=e^{\beta n}
\]
where $\beta=\beta_1$. Consequently $e^{-\beta_k}\eta e^{\beta_{k-|\eta|}}=e^{-\beta|\eta|}\eta$
for all $\eta\in F_E$ and all $k$.
\end{lemma}
\begin{proof}
This is just from the definition: for each multi-index $\rho$ of length $n$, we write
$\rho = (\underline{\rho},\rho_n)$ where $\underline{\rho}$ has length $n-1$, and
calculate:
\begin{align*}
e^{\beta_n}=\sum_{|\rho|=n}\pairL_A(e_\rho|e_\rho)
    &= \sum_{|\rho|=n}\pairL_A(e_{\underline{\rho}}{\pairL_A(e_{\rho_n}|e_{\rho_n})}|e_{\underline{\rho}})
    = \sum_{|\rho|=n}\pairL_A(e_{\underline{\rho}}e^\beta|e_{\underline{\rho}})\\
   & = e^\beta\sum_{|\rho|=n}\pairL_A(e_{\underline{\rho}}|e_{\underline{\rho}})
    = e^\beta e^{\beta_{n-1}}.\qedhere
\end{align*}
\end{proof}

\subsection{The Kasparov module representing the extension}
\label{subsec:kas-mod}

For this section we assume that the bimodule $E$ satisfies the hypotheses of Proposition
\ref{prop:Phi-infty}, so that the expectation $\Phi_\infty:\Oo_E\to A$ is defined.

The functional $\Phi_\infty$ is, by construction, gauge invariant in the sense that if
$\xi\in E^{\ox k}$ and $\eta\in E^{\ox n}$ with $k\neq n$ then $\Phi_\infty(T_\xi
T_\eta^*)=0$. We define an $A$-valued inner product on $\Oo_E$ by
\begin{equation}
\pairR(S_1|S_2)_A:=\Phi_\infty(S_1^*S_2),\qquad S_1,\,S_2\in \Oo_E.
\label{eq:psi-product}
\end{equation}
We observe that $\mathcal{N}=\{T\in \Tt_E:\,\Phi_\infty(T^*T)=0\}$
is an $A$-bimodule: it carries a right action because $\Phi_\infty$ is
$A$-bilinear and it carries a left action because $\Phi_\infty(T^*a^*aT)\leq\Vert
a\Vert^2\Phi_\infty(T^*T)$. Similarly $\mathcal{N}$ is a left $\Tt_E$ module, and as
$\Phi_\infty$ vanishes on $\End_A^0(F_E)$ we have $k\cdot \Tt_E\subset \mathcal{N}$ for
all $k\in \End^0_A(F_E)$. These observations justify the following definition.

\begin{defn}
We let $(\Oo_E)^\Phi_A$ denote the right $C^*$-$A$-module obtained
as the completion of $\Oo_E/\mathcal{N}$ in the norm
$\Vert S+\mathcal{N}\Vert:=\Vert\Phi_\infty(S^*S)\Vert_A$.
We denote the class of $S_\mu S_\nu^*$ in $(\Oo_E)^\Phi_A$ by
$W_{\mu,\nu}$, and as a notational shortcut, we write $W_\mu$ for the class of $S_\mu$
instead of $W_{\mu,1}$.
\end{defn}

Our notational ambiguity will not cause problems: we have written
$\Phi_\infty$ for the functional $\Tt_E \to A$ obtained from
Proposition~\ref{prop:Phi-infty}, and also for the functional on $\Oo_E$ to which this
functional descends. So we can form a Hilbert bimodule either by using the former
$\Phi_\infty$ to define an $A$-valued sesquilinear form on $\Tt_E$, or by using the
latter $\Phi_\infty$ to define one on $\Oo_E$. But since $\Phi_\infty$ vanishes on the
covariance ideal, these two modules coincide, and the canonical representation of $\Tt_E$
on the former induced by multiplication actually descends to the corresponding
representation of $\Oo_E$ on the latter.

In particular, the module $(\Oo_E)^\Phi_A$ carries a representation of $\Oo_E$, defined
by left multiplication, and so, for instance, $S_\mu S_\nu^*\cdot W_{\sigma, \rho} =
W_{\mu\cdot \pairR(e_\nu|e_{\underline{\sigma}})_A\cdot\overline{\sigma}, \rho}$ when
$|\nu|\leq|\sigma|$.

Using the module $(\Oo_E)^\Phi_A$, we can now produce a Kasparov module representing the
extension class $\extcls\ox_{\End^0(F_E)}[F_E]$.

\begin{thm}
If $E_A$ is a finitely generated projective $A$-bimodule satisfying the hypotheses of
Proposition~\ref{prop:Phi-infty}, and $e_1,\dots,e_N$ is a frame for $E_A$, then the
series
$$
\sum_{k\geq 0}\sum_{|\rho|=k}\Theta_{W_{e_\rho},W_{e_\rho}}
$$
converges strictly to a projection $P \in \End_A((\Oo_E)_A^\Phi)$. The map $\xi \mapsto W_\xi$ is
an isometric isomorphism of $F_E$ onto $P (\Oo_E)^\Phi_A$. The left action of $\Oo_E$ on
$(\Oo_E)_A^\Phi$ has compact commutators with $P$. The triple
\[
(\Oo_E,(\Oo_E)^\Phi_A,2P-1)
\]
is a Kasparov module which represents the class $\extcls\ox_{\End^0(F_E)}[F_E]$.
\end{thm}
\begin{proof}
The map $\iota:(F_E)_A\to (\Oo_E)^\Phi_A$ given by $\xi\mapsto\iota(\xi):=W_\xi$ is isometric.
This follows from the computation
$$
(W_\xi|W_\xi)_A=\Phi_\infty(W_\xi^*W_\xi)=\Phi_\infty((\xi|\xi)_A)=(\xi|\xi)_A,
$$
and polar decomposition. We now define projections $P_k$ for $k\geq 0$ by
$$
P_k:=\sum_{|\rho|=k}\Theta_{W_{e_\rho},W_{e_\rho}}.
$$
The $P_k$ are adjointable because they are finite sums of rank one operators
for which $\Theta_{\xi,\eta}^*=\Theta_{\eta,\xi}$, and this formula then shows that
the $P_k$ are self-adjoint. The projection property follows from the computation
\begin{align*}
P_kP_\ell
&=\Big(\sum_{|\rho|=k}\Theta_{W_{e_\rho},W_{e_\rho}}\Big)
\Big(\sum_{|\sigma|=\ell}\Theta_{W_{e_\sigma},W_{e_\sigma}}\Big)
=\sum_{|\rho|=k,\,|\sigma|=\ell}\Theta_{W_{e_\rho}\,(W_{e_\rho}|W_{e_\sigma}),W_{e_\sigma}}\\
&=\left\{\begin{array}{ll} 0 & k\neq \ell\\
\sum_{|\rho|=k}\Theta_{W_{e_\rho},W_{e_\rho}} & k=\ell\end{array}\right.\\
&=\delta_{k,\ell}P_k
\end{align*}
which also shows that the various $P_k$ are mutually orthogonal. From these computations,
it is immediate that $P:=\sum_{k\geq 0}P_k$ is a projection.
Next observe that the image of $P$ is contained in $\iota(F_E)\subset (\Oo_E)^\Phi_A$,
since
\begin{align*}
PW_{\mu, \nu}
    &=\sum_{0\leq|\rho|\leq|\mu|}W_{e_\rho}\Phi_\infty(S_{e_\rho}^*S_{\mu}S_{\nu}^*)
     =\sum_{0\leq|\rho|\leq|\mu|} W_{e_\rho}\Phi_\infty\big(\pairR(e_\rho|{\underline{\mu}})_A S_{{\overline{\mu}}}S_{\nu}^*\big)\\
    &=\lim_{k\to\infty}\sum_{0\leq|\rho|\leq|\mu|}W_{e_\rho}\pairR(e_\rho|{\underline{\mu}})_A \pairL_A({\overline{\mu}}|\nu e^{\beta_{(k-|\nu|)}})e^{-\beta_k}\\
    &= \lim_{k\to\infty}\sum_{0\leq|\rho|\leq|\mu|}W_{e_\rho(e_\rho|{\underline{\mu}})_A}\, \pairL_A({\overline{\mu}}|\nu e^{\beta_{(k-|\nu|)}})e^{-\beta_k}.
\end{align*}
If $|\overline{\mu}| \not= |\nu|$, then  $\pairL_A({\overline{\mu}}|\nu e^{\beta_{(k-|\nu|)}}) =
\Phi_\infty(S_{\overline{\mu}} S^*_{\nu e^{\beta_{k - |\nu|}}}) = 0$, and so we see that
$PW_{\mu,\nu} = 0$ if $|\mu| < |\nu|$, and
\[
PW_{\mu, \nu}
    = W_{\underline{\mu}} \Phi_\infty(S_{\overline{\mu}} S^*_\nu)
        \quad\text{ if $|\mu| \ge |\nu|$ and
        $\mu = \underline{\mu} \otimes \overline{\mu}$ with $|\overline{\mu}| = |\nu|$}.
\]
Thus $PW_{\mu,\nu}\in \iota(F_E)\cdot A\subset\iota(F_E)$. Since the image of $P$ can
easily be seen to contain $W_{e_\rho}$ for all multi-indices $\rho$, it follows that
the image of $P$ equals $F_E$. Thus we also learn that $F_E$ is a complemented sub-module
of $(\Oo_E)^\Phi_A$ and that the isometric inclusion $\iota:F_E\hookrightarrow (\Oo_E)^\Phi_A$
is also adjointable. It is straightforward to check
that the map $\iota$ also intertwines the actions of $\Tt_E$ on these copies of $F_E$. Thus the
compression of the action of $\Oo_E$ on $(\Oo_E)^\Phi_A$ to the subspace
$P(\Oo_E)^\Phi_A$ gives a positive splitting of the quotient map
$\Tt_E\stackrel{q}{\to}\Oo_E$.

To see that $(\Oo_E,(\Oo_E)_A^\Phi,2P-1)$ is a Kasparov
module, we must verify the compactness of the
commutators $[P,S_{\mu}]$.
Fix elementary tensors $\rho,\,\sigma,\,\mu\in F_E$, and observe that
\begin{align*}
PS_{\mu}W_{\rho, \sigma}
    = 0 &\quad\mbox{if}\ |\mu|+|\rho|-|\sigma|<0,\quad\text{ and}\quad
S_{\mu}PW_{\rho, \sigma}
    =0 \quad\mbox{if}\ |\rho|-|\sigma|<0.
\end{align*}
In the following, if $ |\mu|+|\rho|-|\sigma|\geq 0$ then we split
$\mu=\underline{\mu}\otimes\overline{\mu}$ so that $|\overline{\mu}|+|\rho|-|\sigma|= 0$.
Then to complete the proof, first observe the easy relation (proved above)
\[
PW_{\mu}\cdot c=W_{\mu}\cdot\Phi_\infty(c),\quad c\in{\rm span}\{S_\alpha S_\beta^*\,:\,|\alpha|=|\beta|\}.
\]
Then
\begin{align*}
&PS_\mu W_{\rho,\sigma}-S_\mu P W_{\rho,\sigma}\\
&=\left\{\begin{array}{ll}
PW_{\mu\rho,\sigma}-W_{\mu\underline{\rho}}\Phi_\infty(S_{\overline{\rho}}S_\sigma^*)
& |\rho|-|\sigma|\geq 0\\
PW_{\mu\rho,\sigma} & |\rho|-|\sigma|<0\end{array}\right.
=\left\{\begin{array}{ll}
0
& |\rho|-|\sigma|\geq 0\\
W_{\mu\rho,\sigma} & -|\mu|\leq|\rho|-|\sigma|<0\\
0 & |\rho|-|\sigma|< -|\mu|
\end{array}\right.\\
&=\sum_{j=-|\mu|}^{-1}P_j S_\mu W_{\rho,\sigma}
\end{align*}
which is explicitly the action of a finite sum of finite rank operators, and so certainly compact.

Since $(2P-1)^2=1$, $(2P-1)^*=(2P-1)$ and $[2P-1,S_{\mu}S_{\nu}^*]$ is compact for all
vectors $\mu,\nu\in F_E$, we have completed the proof that we obtain an
odd Kasparov module.
So it suffices to show that the Busby invariant agrees with that of the class
$\extcls\ox_{\End^0_A(F_E)}[F_E]$. We have seen that the representation of $\Tt_E$
on $P (\Oo_E)^\Phi_A$ is isomorphic to the Fock representation, so we just need to show
that the representation $\pi : \Oo_E \to \End_A((\Oo_E)^\Phi_A)$ induced by
multiplication is faithful. Since $\Phi_\infty$ is the identity map on $A$, the image of
$A$ in the module $(\Oo_E)^\Phi_A$ is a copy of the standard module ${}_AA_A$, and so
$\pi$ is faithful on $A$. As discussed above, the gauge action on $\Oo_E$ determines a
unitary action of $\T$ on $(\Oo_E)^\Phi_A$, and this unitary action induces an action
$\beta$ of $\T$ on $\End_A((\Oo_E)^\Phi_A)$. It is routine to check that $\beta_z \circ
\pi = \pi \circ \gamma_z$ for all $z$. So the gauge-invariant uniqueness theorem
\cite[Theorem~4.5]{Katsura:TAMS04} shows that $\pi$ is injective.
\end{proof}

\begin{corl}
Let $E$ be a finitely generated bi-Hilbertian $A$-bimodule. Then the boundary maps in the
$K$-theory and $K$-homology exact sequences for Cuntz--Pimsner algebras, labelled
(1)~and~(2) in \cite[Theorem~4.1]{Pimsner}, are given by the Kasparov product with
$(\Oo_E,(\Oo_E)^\Phi_A,2P-1)$.
\end{corl}

\begin{rmk}
If $E$ is a SMEB then the class $(\Oo_E,(\Oo_E)^\Phi_A,2P-1)$ is just the class of the Fock
module presented earlier.
\end{rmk}

\section{KMS functionals for bimodule dynamics}

Given a bi-Hilbertian bimodule $E$ over a unital algebra $A$, we have seen that
the (right) Jones--Watatani
index
$
\Phi(\Id_E)=e^\beta\in \mathcal{Z}(A)
$
carries useful structural information about the associated Cuntz-Pimsner algebra.
The Jones--Watatani index can be defined for a much wider class of bimodules than those
considered here, and we refer to \cite{KajPinWat} for further examples, the general
framework, and relations to conjugation theory.

Our aim in this final section is to show that the
Jones--Watatani index of the bimodule $E$ also determines a
natural one-parameter group of automorphisms of $\Oo_E$ that often admits a natural KMS state.
The dynamics and most of the ingredients of the KMS states we construct arise from the
bimodule alone, but we require one additional ingredient: a state on $A$ which is
invariant for $E$ in an appropriate sense.

\begin{defn}
Let $E$ be a bi-Hilbertian bimodule over a unital $C^*$-algebra $A$. A state
$\phi:A\to\C$ is $E$-invariant if for all $e_1,\,e_2\in E$ we have
$\phi((e_1|e_2)_A)=\phi({}_A(e_2|e_1))$.
\end{defn}

\begin{lemma}
If a state $\phi:A\to\C$ is $E$-invariant, then it is a trace.
\end{lemma}
\begin{proof}
For $e,f \in E$ and $a \in A$, we have
$$
\phi(\pairR(e|f)_Aa)=\phi(\pairR(e|fa)_A)=\phi(\pairL_A(fa|e))=\phi(\pairL_A(f|ea^*))
=\phi(\pairR(ea^*|f)_A)=\phi(a\pairR(e|f)_A).
$$
So in particular $\phi$ is tracial on the range of $\pairR(\cdot|\cdot)_A$. Since the
right inner-product is full, this completes the proof.
\end{proof}

Before proceeding, we present some examples that demonstrate that the existence of an
invariant trace is not a prohibitively restrictive hypothesis.

\begin{exm}[Crossed products by $\Z$]
Let $E$ be the module $A$ with left action given by an automorphism $\alpha$, as in
Example \ref{ex:cross}. Then the definition of an $E$-invariant state
$\phi:A\to \C$ immediately says that $\phi$ is $\alpha$-invariant. When $A=C(X)$ is
abelian, this is of course just an $\alpha$-invariant measure.
\end{exm}

\begin{exm}[Directed graphs] Let $G = (G^0,G^1,r,s)$ be a finite directed graph
with no sinks or sources. Let $A = \C^{|G^0|}$ and let $E=\overline{C_c(G^1)}$ be the
bi-Hilbertian $A$-bimodule from Example \ref{sub:CK-algs}. Define $A \to \C$ by
\[
 \varphi(f) = \sum_{v\in G^0} f(v)
\]
where $f\in C(G^0)$. If $\xi,\eta\in C_c(G^1)$ we have
\begin{align*}
 \varphi((\xi|\eta)_A) &= \sum_{v\in G^0} (\xi|\eta)_A(v)
 = \sum_{v\in G^0} \sum_{s(e) = v} \overline{\xi(e)}\eta(e)
= \sum_{v\in G^0} \sum_{r(e) = v} \eta(e) \overline{\xi(e)}
= \varphi({}_A(\eta|\xi))
\end{align*}
so $\varphi$ is $E$-invariant.
\end{exm}

\begin{exm}[Topological graphs]
Let $G = (G^0,G^1,r,s)$ be a topological graph with $r : G^1 \to G^0$ a local
homeomorphism. Let $E$ be the bi-Hilbertian
bimodule over $A = C_0(G^0)$ from Example \ref{ex:top-graph}.
Suppose that $\mu$ is a probability measure on $G^0$ satisfying
\[
 \int_{r(\supp \xi)} \xi d\mu = \int_{s(\supp\xi)} \xi d\mu
\]
whenever $\xi\in C_c(G^1)$ with $r$ and $s$ bijective on $\supp\xi$. Given $f\in A$ define
\[
 \varphi(f) = \int f d\mu.
\]
Then $\varphi$ is $E$-invariant.
\end{exm}

It is fairly clear that the preceding example can be further generalised to the twisted
topological graph algebras of Li \cite{Li2014} (see Example~\ref{sub:twistedtop-graph}).

We now show how an $E$-invariant trace can be used to construct a KMS state for a
dynamics on $\Oo_E$ determined by the Jones--Watatani index of the module.

\begin{lemma}
\label{lem:dynamics} Let $E$ be a finitely generated bi-Hilbertian $A$-bimodule, and let
$(T, \pi)$ denote the universal generating Toeplitz representation of $E$ in $\Tt_E$.
There is a dynamics $\gamma : \R \to \Aut(\Tt_E)$ such that
$$
\gamma_t(T_e):=\pi(e^{i\beta t}) T_e,\quad e\in E,\quad\text{ and }\quad
\gamma_t(\pi(a)) = \pi(a), \quad a \in A.
$$
Moreover, this dynamics descends to a dynamics, also denoted $\gamma$, on $\Oo_E$.
\end{lemma}
\begin{proof}
Fix $t \in \R$ and define $R : E \to \Tt_E$ by $R_e := e^{i\beta t} T_e$. Then $R$ is a
linear map, and since $e^{i\beta t}$ is central in $A$, we have
\[
\pi(a)R_e = \pi(a e^{i\beta t}) T_e = \pi(e^{i\beta t} a)T_e = R_{a\cdot e}.
\]
We have $R_e \pi(a) = R_{e\cdot a}$ by associativity of multiplication. For $e,f \in E$,
we have $R^*_e R_f = T^*_e \pi(e^{-i\beta t} e^{i\beta t}) T_f$. Since $e^{\beta}$ is
invertible and positive, $e^{i\beta t}$ is unitary, with adjoint $e^{-i\beta t}$, and so
$R^*_e R_f = T^*_e T_f = \pi(\pairR(e|f)_A)$. So $(R, \pi)$ is a Toeplitz
representation.

Now the universal property of $\Tt_E$ shows that there is a homomorphism $\gamma_t :
\Tt_E \to \Tt_E$ satisfying the desired formulae. Clearly $\gamma_s \circ \gamma_t =
\gamma_{s+t}$ and $\gamma_e = \Id_{\Tt_E}$ on generators, and it follows that $t \mapsto
\gamma_t$ is a homomorphism of $\R$ into $\Aut(\Tt_E)$. A routine
$\varepsilon/3$-argument shows that this homomorphism is strongly continuous, completing
the proof.

To see that $\gamma$ descends to a dynamics on $\Oo_E$, observe that with $(R, \pi)$ as
above, for $e,f \in E$, we have
\[
R^{(1)}(\theta_{e,f}) = R_e R^*_f
    = \pi(e^{i\beta t}) T_e T^*_f \pi(e^{-i\beta t}).
\]
So for $a \in A$, writing $\phi(a) \in \End^0_A(E)$ for the compact operator given by
$\phi(a) e = a \cdot e$ for $e \in E$, we have $R^{(1)}(\phi(a)) = \pi(e^{i\beta t})
T^{(1)}(\phi(a)) \pi(e^{-i\beta t})$. Since $e^{i\beta t}$ is a central unitary in $A$,
we have $\pi(a) =  \pi(e^{i\beta t} a e^{-i\beta t})$ for all $a \in A$, and hence
\begin{align*}
\gamma_t(T^{(1)}(\phi(a)) - \pi(a))
    &= R^{(1)}(\phi(a)) - \pi(a) \\
    &= \pi(e^{i\beta t}) T^{(1)}(\phi(a)) \pi(e^{-i\beta t}) - \pi(e^{i\beta t} a e^{-i\beta t}) \\
    &= \pi(e^{i\beta t}) \big(T^{(1)}(\phi(a)) - \pi(a)\big) \pi(e^{-i\beta t}).
\end{align*}
So each $\gamma_t$ preserves the covariance ideal and therefore descends to $\Oo_E$
as claimed.
\end{proof}

Note that, in general, $e^{i\beta t}f\neq fe^{i\beta t}$ for $f\in E$. So we typically
have
$$
    \gamma_t(S_{e_\rho}S_{e_\gamma}^*) \neq e^{i\beta|\rho|t} S_{e_\rho}S_{e_\gamma}^*e^{-i\beta|\gamma|t}.
$$

The dynamics on $\Tt_E$ described in Lemma~\ref{lem:dynamics}  is implemented by the second quantisation of the one parameter
unitary group $t\mapsto U_t=e^{i\beta t}$, \cite{LacaNeshveyev}. The second quantisation
is given by $\Gamma(U_t)=\Id_A\oplus\, U_t\oplus(U_t\ox U_t)\oplus\cdots$. We let
$$
\D=\oplus_{n\in\N} \big(\beta\otimes{\rm Id}_{E^{\otimes n-1}}
+{\rm Id}_E\otimes\beta\otimes{\rm Id}_{E^{\otimes n-2}}+\cdots+
{\rm Id}_{E^{\otimes n-1}}\otimes\beta\big)
$$
be the (self-adjoint, regular) generator of the unitary group $\Gamma(U_t)$. Combining
ideas from \cite{CNNR,LacaNeshveyev} we can construct a KMS state for $\gamma$. Recall
from \cite[Theorem~1.1]{LacaNeshveyev} that if $\phi$ is a trace on $A$, and $M$ is a
right-Hilbert $A$-module, then there is a norm lower semicontinuous semifinite trace
$\Tr_\phi$ on $\End^0(M)$ such that
\begin{equation}\label{eq:trphi}
\Tr_\phi(\Theta_{\xi,\eta})=\phi(\pairR(\eta|\xi)_A)\quad\text{ for all $\xi,\eta \in M$.}
\end{equation}
Note that if $M$ is finitely generated, then $\Tr_\phi$ is finite.

\begin{prop}\label{prp:KMS}
Let $E$ be a finitely generated bi-Hilbertian $A$-bimodule, $\phi$ an $E$-invariant trace
on $A$, and $\beta\in \mathcal{Z}(A)$ as defined in Equation \eqref{eq:log-beta}. Let $N$
denote the number operator on the Fock space. Then there is a state $\phi_\D$ on $\Tt_E$ such
that
\[
\phi_\D(T_\xi T_\eta^*)=\res_{s=1}\Tr_\phi(e^{-\D}T_\xi T_\eta^*(1+N^2)^{-s/2}) \quad\text{ for all $\xi,\eta \in \textstyle\bigcup_n E^{\otimes n}$.}
\]
This $\phi_\D$ vanishes on $\End^{00}(F_E)$, and descends to a linear functional, still
denoted $\phi_\D$ on $\Oo_E$. Moreover, $\phi_\D$ is a KMS$_1$-state of $\Oo_E$ for
$\gamma$.
%
%For $S_1,\,S_2\in \lsp\{S_\xi S^*_\eta : \xi,\eta \in \bigcup_n E^{\otimes n}\}$, we have
%$$
%\phi_\D(S_1S_2)=\phi_\D(\gamma_{\sqrt{-1}}(S_2)S_1),
%$$
%and so is KMS for $\gamma$ at temperature 1.
\end{prop}
\begin{proof}
Let $\Tr_\phi$ be the functional obtained from~\eqref{eq:trphi} with $M = F_E$, and for
each $k$, let $\Tr_{\phi,k}$ be the functional obtained from~\eqref{eq:trphi} with $M =
E^{\otimes k}$. If $\eta \in E^{\otimes k}$ and $\zeta \in E^{\otimes l}$ with $k \not=
l$, then $\pairR(\eta|\xi)_A = 0$ in $F_E$, and so~\eqref{eq:trphi} gives
$\Tr_\phi(\Theta_{\xi,\eta}) = 0$; and if $\xi,\eta\in E^{\ox k}$ then~\eqref{eq:trphi}
gives $\Tr_\phi(\Theta_{\xi,\eta}) = \Tr_{\phi,k}(\Theta_{\xi,\eta})$.

For each $n$, let $P_n\in \End_A(F_E)$ be the projection onto $E^{\otimes n}$. For
$\xi,\eta \in \bigcup_k E^{\otimes k}$, we have
$$
e^{-\D}T_\xi T_\eta^*(1+N^2)^{-s/2}
=\sum_{n,m=0}^\infty P_ne^{-\D}T_\xi T_\eta^*(1+m^2)^{-s/2}P_m,
$$
and so
$$
\Tr_\phi(e^{-\D}T_\xi T_\eta^*(1+N^2)^{-s/2})
=\sum_{n=0}^\infty \Tr_{\phi,n}(P_ne^{-\D}T_\xi T_\eta^*(1+n^2)^{-s/2}P_n).
$$
Since $\xi,\eta \in E^{\otimes k}$,
$$
\Tr_{\phi,n}(P_ne^{-\D}T_\xi T_\eta^*P_n)=\left\{\begin{array}{ll} 0 & n<k\\
\Tr_{\phi,n}(e^{-\D}\Theta_{\xi,\eta}\ox\Id_{E^{\ox n-k}}) & n\geq k\end{array}\right.
$$
Fix $n \ge k$, let $\{e_\rho\}$ be a frame for $E^{\ox n-k}$, and compute:
\begin{align*}
\Tr_{\phi,n}(e^{-\D}\Theta_{\xi,\eta}\ox\Id_{E^{\ox n-k}})
    &=\sum_{|\rho|=n-k}\Tr_{\phi,n}(e^{-\D}\Theta_{\xi\ox e_\rho,\eta\ox e_\rho})\\
    &=\sum_{|\rho|=n-k}\phi\big(\pairR(\eta\ox e_\rho|e^{-\D}\xi\ox e_\rho)_A\big)\\
    &=\sum_{|\rho|=n-k}\phi\big(\pairR(e_\rho|{\pairR(\eta|e^{-\D}\xi)_A}e^{-\D}e_\rho)_A\big)\\
    &=\sum_{|\rho|=n-k}\phi\big(\pairR(\otimes^{n-k}_{j=1} e^{-\beta/2}e_{\rho_j}|
        {\big(\eta\mid e^{-\D}\xi)_A}\cdot\big(\otimes^{n-k}_{j=1} e^{-\beta/2}e_{\rho_j}\big))_A\big)\\
    &=\sum_{|\rho|=n-k}\phi\big({}_A\big({\pairR(\eta|e^{-\D}\xi)_A}\cdot
        \big(\otimes^{n-k}_{j=1} e^{-\beta/2}e_{\rho_j}\big)\mid
        \otimes^{n-k}_{j=1} e^{-\beta/2}e_{\rho_j}\big)\big)\\
    &=\sum_{|\rho|=n-k}\phi\big(\pairR(\eta|e^{-\D}\xi)_A\,\pairL_A(
        \otimes^{n-k}_{j=1} e^{-\beta/2}e_{\rho_j}|
        \otimes^{n-k}_{j=1} e^{-\beta/2}e_{\rho_j})\big).
\end{align*}
We have $\sum \pairL_A(\otimes^{n-k}_{j=1} e^{-\beta/2}e_{\rho_j} | \otimes^{n-k}_{j=1}
e^{-\beta/2}e_{\rho_j}) = 1_A$ by the calculations of Lemma~\ref{lem:correct-beta1}.
Hence
\[
\Tr_{\phi,n}(e^{-\D}\Theta_{\xi,\eta}\ox\Id_{E^{\ox n-k}}) = \phi((\eta|e^{-\D}\xi)_A).
\]
Hence
$$
\Tr_\phi(e^{-\D}T_\xi T_\eta^*(1+N^2)^{-s/2})
=\sum_{n=k}^\infty \phi(\pairR(\eta|e^{-\D}\xi)_A)(1+n^2)^{-s/2},
$$
and we see that $\phi_\D(T_\xi T_\eta^*) := \res_{s=1}\Tr_\phi(e^{-\D}T_\xi
T_\eta^*(1+N^2)^{-s/2})$ is well-defined, and
\begin{equation}
\phi_\D(T_\xi T_\eta^*) = \phi(\pairR(\eta|e^{-\D}\xi)_A).
\label{eq:i-can-count}
\end{equation}

Fix $a \in A$. By the calculation of Lemma~\ref{lem:correct-beta1}, and centrality of
$\beta$,
\begin{align}
\phi_\D(a)&=\phi_\D(a1_{\Tt_E})\nonumber\\
&=\res_{s=1}\sum_n\Tr_{\phi,n}(e^{-\D}a{\rm Id}_{E^{\otimes n}})(1+n^2)^{-s/2}\nonumber\\
&=\res_{s=1}\sum_n\phi\Big(\sum_{|\rho|=n}(e_\rho|e^{-\D}ae_\rho)_A\Big)(1+n^2)^{-s/2}\nonumber\\
&= \res_{s=1}\sum_n\phi\Big(\sum_{|\rho|=n}
{}_A\Big(a\cdot\big(\otimes^{n}_{j=1} e^{-\beta/2}e_{\rho_j}\big)\mathbin{\Big|} \otimes^{n}_{j=1} e^{-\beta/2}e_{\rho_j})\Big)(1+n^2)^{-s/2}\nonumber\\
&= \res_{s=1}\sum_n\phi(a1_A)(1+n^2)^{-s/2}= \phi(a).
\label{eq:more-counting}
\end{align}

To check that $\phi_\D$ extends to a norm-decreasing linear map on $\Tt_E$,
apply~\eqref{eq:more-counting} to $a = 1_A\in A$ to see that
$\phi_\D(1_{\Tt_E})=\phi(1_A)=1$. Equation~\eqref{eq:i-can-count} shows that the formula
$\sum_i T_{\xi_i} T^*_{\eta_i} \mapsto \sum_i \phi_\D(T_{\xi_i} T^*_{\eta_i})$ carries
positive elements of $\lsp\{T_\xi T^*_\eta : \xi,\eta \in \bigcup_k E^{\otimes k}\}$ to
$[0,\infty)$. Hence, for $T = \sum_i T_{\xi_i} T^*_{\eta_i}$ self-adjoint,
$|\phi_\D(T)|\leq \Vert T\Vert \phi_\D(1_{\Tt_E}) = \|T\|$. So the formula
$\phi_\D(\sum_i T_{\xi_i} T^*_{\eta_i})$ is well-defined and bounded, so extends to a
bounded linear functional on $\Tt_E$ satisfying $\phi_\D(1) = 1$; that is, a state.

A calculation like~\eqref{eq:i-can-count} shows that
$\phi_\D\big(a\sum_{j=1}^N\Theta_{e_j,e_j}\big) = \phi(a)$ as well. So $\phi_\D$ vanishes
on the covariance ideal, and hence descends to a state on $\Oo_E$.

We check the KMS condition. For $S_1,\,S_2 \in \clsp\{S_\xi S_\eta^* : \xi,\eta \in
\bigcup_k E^{\otimes k}\}$, we have
$$
\Tr_\phi(e^{-\D}S_1S_2(1+N^2)^{-s/2})=\Tr_\phi(e^{-\D}(e^{\D}S_2e^{-\D})S_1(1+N^2)^{-s/2}+
e^{-\D}S_1[S_2,(1+N^2)^{-s/2}]).
$$
Now we show that the commutator $[S_2,(1+N^2)^{-s/2}]$ is `trace-class'. For $0<s<2$
we employ the integral formula for fractional powers to find
\begin{align*}
[S_2,(1+N^2)^{-s/2}]
&=\frac{\sin(s\pi/2)}{\pi}\int_0^\infty
\lambda^{-s/2}(1+\lambda+N^2)^{-1}[N^2,S_2](1+\lambda+N^2)^{-1}d\lambda.
\end{align*}
We claim that the integral on the right converges in norm for all $2>s>0$, proving finiteness
of the sum defining $\Tr_\phi(e^{-\D}S_1[S_2,(1+N^2)^{-s/2}])$ at $s=1$.
To see this, observe that gauge invariance says that we
need only consider the case when $S_1S_2$
is homogenous of degree $0$ for the
gauge action. If $S_2$ is of degree zero, then there is nothing to
prove. For $S_2$ homogenous of degree
$m$ we have
\begin{align*}
\Tr_\phi&(e^{-\D}S_1[S_2,(1+N^2)^{-s/2}])
  = \sum_{n=0}^\infty\Tr_{\phi,n}( P_ne^{-\D}S_1[S_2,(1+N^2)^{-s/2}]P_n)\\
    &=\sum_{n=0}^\infty \Tr_{\phi,n}\Big(P_ne^{-\D}S_1\frac{\sin(s\pi/2)}{\pi}\\
   &\hskip6em{}  \times \int_0^\infty \lambda^{-s/2}
      (1+\lambda+N^2)^{-1}(N[N,S_2]+[N,S_2]N)(1+\lambda+N^2)^{-1}P_n\,d\lambda\Big)\\
   &=m\sum_{n=0}^\infty
        \Tr_\phi\bigg(P_ne^{-\D}S_1S_2P_n\bigg)\frac{\sin(s\pi/2)}{\pi}f_n(s)
\end{align*}
where
$$
f_n(s)\!=\!\!\int_0^\infty\!\!\!\!\!\lambda^{-s/2}
\Big((1+\lambda+(n+m)^2)^{-1}(n+m)(1+\lambda+n^2)^{-1}+ (1+\lambda+(n+m)^2)^{-1}n(1+\lambda+n^2)^{-1}\Big)\,d\lambda.
$$
Now observe that
$\Big| \Tr_\phi\Big(P_ne^{-\D}S_1S_2P_n\Big)\Big|<\Vert S_1S_2\Vert$ for all $n$, and use
the elementary estimates
$$
(1+\lambda +n^2)^{-1}n\leq \frac{1}{\sqrt{2}}(1+\lambda)^{-1/2},\quad
(1+\lambda +n^2)^{-1}\leq(1/2+\lambda)^{-\epsilon} (1/2+n^2)^{-(1-\epsilon)}\ \mbox{for}\ 1>\epsilon>0
$$
to see that
$$
f_n(s)\leq \frac{1}{\sqrt{2}}\Big((1/2+n^2)^{-(1-\epsilon)} +(1/2+(n+m)^2)^{-(1-\epsilon)}\Big)
\int_0^\infty \lambda^{-s/2}(1+\lambda)^{-1/2}(1/2+\lambda)^{-\epsilon}\,d\lambda.
$$
These estimates, along with a suitable choice of $\epsilon$,
show that $\Tr_\phi(e^{-\D}S_1[S_2,(1+N^2)^{-s/2}])$ is finite for all $2>s>0$, and
more generally for $2>\Re(s)>0$.

Taking the derivative with respect to
$s$ of the function $s\mapsto \Tr_\phi(e^{-\D}S_1[S_2,(1+N^2)^{-s/2}])$ and repeating the estimate
(now with an extra $\log(\lambda)$ in the integral defining $f_n(s)$)
shows that the derivative is also finite at $s=1$, proving holomorphicity. Hence the residue at $s=1$ of
$s\mapsto \Tr_\phi(e^{-\D}S_1[S_2,(1+N^2)^{-s/2}])$ vanishes and $\phi_\D$ is KMS.
\end{proof}
\begin{rmk}
It is tempting to  consider  a dynamics on $\Oo_E$ coming from the
unitary group $W_t=\oplus_{k\geq 0} e^{i\beta_k t}$, where
$\beta_k=\log(\Phi_k(\Id_{E^{\ox k}}))$. There is a dynamics $\sigma$ on $\End_A(F_E)$
given by $\sigma_t(T) = W_t T W^*_t$, and it is natural to ask whether $\sigma$ restricts
to a dynamics on $\Tt_E$.

Observe, however, that since $e^{\beta_0} = 1_A$ and $e^{\beta_1} = e^\beta$, the
dynamics $\sigma$ agrees on the generators of $\Tt_E$ with the dynamics defined in
Lemma~\ref{lem:dynamics}. So if $\sigma$ does indeed extend to $\Tt_E$, then it agrees
with $\gamma$, and the analysis above applies. That is, there is nothing new to be gained
by considering the dynamics $\sigma$, at least for the algebra $\Tt_E$.
\end{rmk}

Proposition~\ref{prp:KMS} combined with the results of \cite{CNNR} yields the following.

\begin{corl}
Let $E$ be a bi-Hilbertian module over $A$ and $\phi:A\to\C$ an $E$-invariant state. Let
$\H=L^2(\Oo_E,\phi_\D)$ be the GNS space of $\phi_\D:\Oo_E\to\C$, and
$\mathcal{N}\subset\B(\H)$ the weak closure of the algebra generated by $\Oo_E$ and the
spectral projections $P_k$ for the unitary extension of the dynamics $\gamma$ to $\H$.
Then $(\Oo_E,\H,\D,\mathcal{N},\phi_\D)$ is a modular spectral triple.
\end{corl}

\end{document}